\documentclass[11pt]{amsart}

\usepackage[margin=1in]{geometry}
\usepackage{amsmath, amssymb, amsthm}
\usepackage{mathtools}
\usepackage{hyperref}
\usepackage{enumitem}
\usepackage{booktabs}
\usepackage{longtable}
\usepackage{array}
\usepackage{microtype}
\usepackage{listings}
\usepackage{xcolor}

\hypersetup{
  colorlinks=true,
  linkcolor=blue,
  citecolor=blue,
  urlcolor=blue
}

\newtheorem{conjecture}{Conjecture}
\newtheorem{definition}{Definition}

\newtheorem{lemma}{Lemma}

\title{Evolving Ranking Functions for Canonical Blow-Ups 
in Positive Characteristic}
\author{Gergely B\'erczi}
\address{Aarhus University}
\email{gergely.berczi@math.au.dk}

\lstdefinestyle{py}{
  language=Python,
  basicstyle=\ttfamily\small,
  keywordstyle=\color{blue!70!black},
  commentstyle=\color{green!40!black},
  stringstyle=\color{red!60!black},
  showstringspaces=false,
  breaklines=true,
  columns=fullflexible,
  frame=single,
  framerule=0.3pt
}

\begin{document}

\begin{abstract}
Resolution of singularities in positive characteristic remains a long-standing open problem in algebraic geometry. In characteristic zero, the problem was solved by Hironaka in 1964, work for which he was awarded the Fields Medal. Modern proofs proceed by constructing suitable ranking functions, that is,  invariants shown to strictly decrease along canonical sequences of blow-ups, ensuring termination. In positive characteristic, however, no such general ranking function is known: Frobenius-specific pathologies, such as the kangaroo phenomenon, can cause classical characteristic-zero invariants to plateau or even temporarily increase, presenting a fundamental obstruction to existing approaches.
In this paper we report a sequence of experiments using the evolutionary search model AlphaEvolve
\cite{NovikovEtAl2025AlphaEvolve}, designed to discover candidate ranking functions for a toy canonical blow-up process.

Our test benchmarks consist of carefully selected hypersurface singularities in dimension $4$ and characteristic
$p=3$, with monic purely inseparable leading term $z^3$, a regime in which naive order-based invariants often fail. After iteratively refining the experimental design, we obtained
a discretized five-component lexicographic ranking function satisfying a bounded-delay descent criterion with zero violations across the benchmark. These experiments in turn motivated our main results: the conjectural delayed ranking functions in characteristic $3$ formulated in Conjectures \ref{conj:clean-lex-weierstrass} and \ref{conj:disc-robust}.

\end{abstract}
\maketitle

\section{Introduction}
Existence of resolution of singularities is one of the central problems of geometry. Systems of polynomial equations in several variables appear in many areas of science and engineering. If such a system involves $n$ variables, its set of solutions forms a subset of $\mathbb{R}^n$, or more generally $k^n$ if we work over the field $k$. In algebraic geometry, this set is called a \emph{variety} and it can have a very sophisticated geometric structure.
A major difficulty in working with these systems is that their solution sets may contain points--so-called \emph{singular points}--where the local geometry is ill-behaved.  A fundamental problem in geometry is to modify the solution set in a controlled way so that these singular points are removed without changing the global geometry significantly. This procedure is known as \emph{resolution of singularities}. One of the main tools used for this purpose is the iterative use of a simple local modification  called \emph{blowing up}.
As a toy example, consider the plane curve defined by the equation
\[
y^2 - x^3 + x^2 = 0.
\]
This curve has a singular point at the origin $(0,0)$, called a node (or double point). The blowing-up at the origin in this simple situation means substitution $y = x t$ which gives
\[
x^2(t^2 - x - 1) = 0.
\]
This equation describes two components: the line $x = 0$ (shown in green, called the \emph{exceptional divisor}) and the smooth curve $t^2 - x - 1 = 0$ (shown in black, called the \emph{proper transform} of the original curve). The black curve is non-singular and represents the resolved version of the original curve.
Geometrically, this blow-up replaces the singular point at the origin by all possible tangent directions at that point, which form the projectivized tangent space of $ \mathbb{R}^2$ at the origin. Note that for this toy example the other substitution $y=xt$ does not add extra points and we can drop that "affine chart". The origin is called the \emph{center} of the blow-up. 

\begin{figure}[ht!]
\centering
\includegraphics[width=40mm]{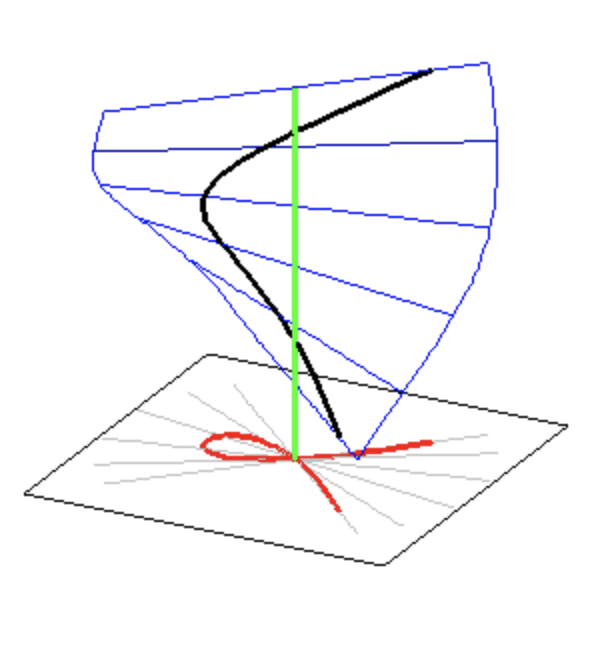}
\caption{A nodal singularity and its resolution. Image by Dona Arapura, https://www.math.purdue.edu/~arapura/graph/nodal.html}
\label{fig1}
\end{figure}

More generally, given an integral scheme $X$ of finite type over a field $k$, the resolution problem asks for a proper birational morphism $\pi\colon Y\to X$ with $Y$ smooth/regular. One aims for a canonical resolution algorithm, given by an iterative sequence of blow-ups with (possibly smooth) centers, where the choice of each center depends only on the intrinsic singularities of $X$ and on the accumulated exceptional divisor data. 

In characteristic zero the existence of such a $\pi$ and iterative process was proved by Hironaka \cite{Hironaka1964}. This celebrated breakthrough result was a major part of the work for which he received the Fields Medal (1970). Hironaka's original argument is essentially existential and highly non-constructive: it does not provide a canonical, purely intrinsic ranking function that determines the centers from local singularity data alone, and it necessarily involves delicate bookkeeping of the accumulating exceptional divisor, i.e. the history of the blowups.

Over the subsequent decades this existence theorem has been refined into a family of canonical and functorial desingularization algorithms.
These algorithms are driven by well-founded resolution functions (also called invariants
or ranking functions) built from intrinsic local data: orders of ideals, coefficient ideals,
Hilbert-Samuel data, and the combinatorics of exceptional divisors.
Without trying to be exhaustive, we refer to the canonical constructions of Bierstone-Milman~\cite{BierstoneMilman1997},
the constructive approach of Villamayor~\cite{Villamayor1989} and its refinements,
and to W\l odarczyk's simplification of the overall structure~\cite{Wlodarczyk2005}, and the recent new simple invariant found by Abramovich-Temkin-W\l odarczyk \cite{AbramovichTemkinWlodarczyk2024,Abramovich2018ICM}.

In positive characteristic the picture is far less complete. Resolution by successive blow-ups is established only in low dimensions. Lipman proved desingularization for excellent surfaces \cite{Lipman1978}, and Cossart-Jannsen-Saito later obtained functorial embedded and non-embedded resolution for arbitrary excellent two-dimensional schemes \cite{CJS2009}. In dimension three, the program of Cossart-Piltant culminates in a resolution theorem for quasi-excellent schemes in arbitrary characteristic (see, e.g., \cite{CossartPiltant2008,CossartPiltant2009,CossartPiltant2014}). In dimension $\ge 4$ the resolution problem remains open, and a decreasing ranking function based on intrinsic invariants, comparable to the characteristic-zero algorithms is still not known.

In positive characteristic most proposed strategies still aim for a canonical blow-up process: one repeatedly blows up smooth centers selected by a complexity measure, with the goal of reaching a monomial (combinatorial) phase where the remaining steps are governed by a normal crossings divisor. The characteristic-zero inductive toolkit, however, breaks down in several ways: hypersurfaces of maximal contact may fail to exist even in low dimension, Frobenius effects interfere with differential constructions, and secondary invariants built from coefficient ideals or residual orders can increase under canonical blow-ups \cite{Hauser2008,Hauser2010,HauserPerlega2018}. Existing approaches attempt to circumvent these obstacles, for instance via elimination-theoretic induction using generic projections \cite{Villamayor1989,Villamayor1995}, or via the Idealistic Filtration Program of Kawanoue-Matsuki, which formulates an explicit blow-up procedure conjectured to terminate and reduces the general case to a monomial case \cite{KawanoueMatsuki2006,KawanoueMatsuki2015}. A weaker resolution result, valid in all dimensions, is known as de Jong's alterations theorem. It produces a generically finite morphism from a regular scheme, though not a birational resolution \cite{deJong1996}. 
The experiments in this paper explore the following question:

\begin{quote}
Can one discover, even in a simplified setting, a canonical blow-up process and a ranking function built from \emph{intrinsic} numerical data of singularities that induces finite termination of the process, despite the possibility of temporary increases?
\end{quote}

\medskip
\noindent\textbf{Methodology and results.}
Our approach is computational. We design a toy canonical blow-up simulator consisting of 
\begin{enumerate}
    \item a test benchmark of carefully selected purely inseparable hypersurface singularities in characteristic $p$,
    \item  a canonical choice of blow-up centers at each step and
    \item a carefully selected collection of intrinsic invariants of the singularities (or suitable proxies, when encoding the invariants directly would significantly slow down the evaluation) that serve as features for a ranking function.The feature set deliberately targets invariants that appear in the positive-characteristic literature.
\end{enumerate}

Using this fixed simulator, we design an AlphaEvolve experiment to discover explicit ranking functions that depend only on the simulator features and evolve a delayed ranking function that is perfect on the chosen test benchmark.

For each experiment, the simulator is fixed in advance: the center-selection rule, the feature set, and the testing benchmark are all predetermined. The simulator follows a single deterministic rule for choosing the next blow-up center, and AlphaEvolve is used solely to search for ranking functions, expressed in terms of the fixed features, that provides delayed  for this given process \cite{NovikovEtAl2025AlphaEvolve,AlphaEvolve2025Scale}.

We performed a sequence of experiments that led to a progressive evolution of the simulator. After each experiment, we manually updated and refined the test benchmark and the feature set (including their proxies), while keeping the center-selection mechanism unchanged. A natural next step would be to make the center-selection rule itself learnable and subject to optimization. Allowing the system to explore alternative blow-up policies could reveal qualitatively new behaviors and provide stronger empirical evidence for termination mechanisms.

After an initial broad benchmark proved to be too ambitious (see Section \ref{section:experimentI}), we restricted our attention to a focused family in
ambient dimension $4$ and characteristic $p=3$ (Appendix A).
Within this focused regime we obtained three evolved ranking functions:
\begin{enumerate}
\item an optimized \emph{continuous} two-component ranker achieving zero violations on a benchmark of $71$ singularities (Section \ref{subsec:experiment1}).
\item a \emph{discretized} five-component ranking function also achieving zero violations on the same benchmark (Section \ref{subsec:discretize}). Discretization is crucial: it implies finite termination.  
\item after discovering a counterexample to an initial structural conjecture for the discretized ranking function
(Theorem~\ref{thm:counterexample}), we extended the benchmark to $100$ adversarial instances, including the found counterexamples. We found a 
new five-component ranker which achieves zero violations on all $100$ instances. 
(Section \ref{subsec:disc-100}).
\end{enumerate}
The evolution of experiments suggested two conjecture about the existence of ranking functions for a restricted class of positive
characteristic singularities, formulated in Conjecture \ref{conj:clean-lex-weierstrass} and
Conjecture~\ref{conj:disc-robust}.

\subsection*{How these experiments were run}
\label{sec:how-run}

This is not a traditional mathematics paper.
Rather, this is a report written by a mathematician for mathematicians, describing a possible way to use AlphaEvolve as an experimental engine to investigate a central problem in birational geometry.
As a geometer with some prior experience in machine learning and reinforcement learning over the last few years, I was keen to explore this territory, but I also found myself repeatedly outside of my comfort zone: the work consisted of designing a computational analogue of the resolution process, proposing candidate invariants and their proxies, iterating on a scoring system that rewards the behavior one expects from a terminating resolution program, and iteratively analyze the outcome of the experiements and make the necessary improvements of the simulator for the next run. 
Here is my work pipeline:
\begin{enumerate}
  \item Curate a test set of singularities.
  The work started with the careful collection of families of positive-characteristic singularities designed to exhibit the main pathologies (purely inseparable leading forms, kangaroo-style secondary jumps, and tie-breaking instability).
  The test was developed iteratively: whenever a learned ranker passed too easily, we went back and added adversarial instances.

  \item Select a dictionary of invariants.
  Next, we fixed a list of $26$ numerical invariants of singularities (see the list in Section \ref{subsec:features}) that are motivated by the characteristic-zero theory and by positive-characteristic phenomena. 
  We implemented these in the AlphaEvolve evaluation harness by explicit proxy computations on monomials; each proxy was checked on examples.

  \item Vibe-code the harness and evolve blocks.
  The evaluation harness and the evolve blocks were vibe-coded using a mixture of large language models (notably AlphaEvolve UI's built-in Gemini 3 engine and GPT-5.2), and then repeatedly checked and edited for basic invariance properties, and to eliminate errors.
  
  \item Iterate through experimental runs.
  Over the course of this project I ran roughly eleven experiments, with runtimes ranging from $3$ hours to $2$ days.  Careful analysis of each run provided a concrete learning path: before the next run we improved the scoring system, identified weaknesses in the test set, repaired several feature proxies, and (crucially) refined the notion of a delayed ranking function to accommodate long plateaux typical in characteristic $p$.
\end{enumerate}

The paper should be read as an exploration of how one can design and evolve resolution experiments rather than as a finished mathematical theory. It was partly motivated by our work \cite{BercziFanZeng} which developed reinforcement learning models to find optimal resolutions in characteristic zero.  Recent works \cite{AEBruhat,AlphaEvolve2025Scale} indicate that AlphaEvolve can serve as a productive engine for mathematical exploration, generating candidate constructions that can then be verified and, in favorable cases, promoted to rigorous theorems.

\noindent \textbf{Declaration of use of AI}
This is an experimental mathematics paper, and a variety of AI-based tools were used extensively throughout the project. The AlphaEvolve user interface relies on gemini-2.5-pro and gemini-3-pro models to generate summaries of the problem setup, describe candidate strategies, suggest and write the evolve blocks and evaluation harnesses. These capabilities were used extensively during the exploratory and iterative phases of the project. We also used code generation outside the AlphaEvolve system for vibe-coding evaluation harnesses, simulator stubs, and for generating and expanding synthetic test sets of singularities. AlphaEvolve automatically produces detailed, Gemini-generated experiment reports, including performance summaries and qualitative analyses of evolved programs. These reports served as a starting point for several sections of this paper; however, the mathematical conclusions, conjectures, and counterexamples presented here arise from independent mathematical analysis. Except as otherwise indicated, the contents of this paper are the work of the author.

\subsection*{Acknowledgements}
I would like to thank the Google DeepMind AlphaEvolve team for the opportunity to participate in the inaugural Trusted Users group with early access to AlphaEvolve, as well as for their continued support throughout this project. I am particularly grateful to Adam Zsolt Wagner, who served as my primary point of contact and provided exceptionally helpful guidance on the UI platform. I am deeply appreciative of this mathematical journey, through which I gained substantial insight into the vast subject of resolutions in positive characteristic, and I benefited greatly from the many discussions and exchanges within the AlphaEvolve community of testers and mathematicians.
\section*{AlphaEvolve as an experimental engine for mathematicians}
\label{sec:alphaevolve}

The experiments reported here were run inside the evolutionary search engine \emph{AlphaEvolve} \cite{NovikovEtAl2025AlphaEvolve,AlphaEvolve2025Scale}.
Since this paper is addressed primarily to mathematicians, we briefly summarize the AlphaEvolve system and the parts of the workflow that were most relevant to us. 

At a high level, AlphaEvolve is a search engine in the space of python scripts, guided by a user-defined score function.
It can be viewed as a reinforcement-learning (RL) environment whose \emph{state} is a Python program, whose \emph{actions} are edits to a designated region of that program,
and whose \emph{reward} is a user-defined score produced by a fixed evaluation harness.
Concretely, the user supplies:
\begin{enumerate}[label=(\roman*)]
    \item a \emph{fixed evaluation harness}, i.e.\ a deterministic procedure that simulates and scores the mathematical task of interest (thereby defining the reward signal), and
    \item an \emph{evolve block}, i.e.\ a delimited region of code that the system is permitted to modify and evolve (thereby defining the action space),
    \item an LLM prompt used to generate candidate edits and new code variants within the evolve block (thereby defining a proposal distribution over actions).
\end{enumerate}
AlphaEvolve then iteratively proposes edits to the evolve block, runs the harness on the resulting script, and updates a memory of high-reward candidates for subsequent exploration.
In this sense, the method implements an outer search loop driven by reward, together with an inner loop in which an LLM generates code mutations.

For mathematicians, the key conceptual take-away is that the object being optimized is not a statement but a policy (i.e python script) whose quality is measured only through the scoring harness.
As a result, correctness is inseparable from the specification of the scoring function.

In our setting the evaluation harness is a static simulator of a canonical blow-up process on a finite list of test singularities.
The evolve block is a pure function \texttt{ranking\_function(features)} producing a tuple-valued rank from a fixed feature vector (i.e vector formed by the numerical invariants of the singularity).
AlphaEvolve's separation of concerns principle is crucial: the harness (simulation, features, scoring) is fixed across evolutionary iterations, while the evolve block is the only part which changes.
This reduces the number of ways in which the system can accidentally cheat the objective.
Two practical rules were essential.
\begin{itemize}
  \item \emph{Input hygiene.}
The evolve block receives only a bounded, deterministic feature vector computed from the \emph{current} simulator state. In particular, it has no access to global randomness or to the full history of the run. A straightforward way for a ranking function to ``cheat'' (which we indeed encountered in the first few runs) is to exploit \emph{time-dependent} signals inadvertently present in the state, such as the number of exceptional divisors created so far. We enforced purity via a determinism test: identical inputs are required to produce identical outputs.

\item \emph{State persistence.}
In some AlphaEvolve modes one can allow information to persist across generations via parent outputs (e.g.\ \texttt{use\_parent\_outputs=True}).
For the experiments in this paper we intentionally disabled any such mechanism inside the ranker itself: the ranker is memoryless, and all state is carried by the fixed simulator.

  \end{itemize}

AlphaEvolve optimizes whatever objective is measured by the harness.
Therefore, the meaning of a learned ranking function is inseparable from:
(i) the test set,
(ii) the feature set,
(iii) the penalty terms and tie-breakers used in scoring, and
(iv) the stopping criteria.
The most time-consuming part of this project was not tuning the model, but iterating on these design choices until they matched the intended mathematical goal (strict descent only after geometric drops; bounded-delay descent across plateaux; $0$-normalization in monomial phase; etc.).

Finally, the AlphaEvolve UI has a built-in visualisation functionality, helping in providing human-readable analysis. In our runs we included a visualization routine producing rank trajectories on a reference instance, and we logged per-stage violation counts, plateau lengths, and increase counts.
This made it feasible to detect when an apparent perfect score was actually an artifact of a weak test suite or a defect of the scoring system.

\section{Background and motivation}\label{sec:background}

The termination of any resolution algorithm is guaranteed by the existence of a \emph{ranking function} (also called resolution function or Lyapunov function)
\[
\rho : \{\text{local singularity data}\}\longrightarrow \Gamma
\]
with values in a well-ordered set $\Gamma$ (i.e a totally ordered set in which every non-empty subset has a minimal element, typically a finite lexicographic product of $\mathbb{N}$), such that the chosen blow-up centers are cut out by loci of
maximal value of $\rho$, and such that after performing the blow-up the value of $\rho$
strictly decreases at every point lying above the center.
Since $\Gamma$ admits no infinite strictly descending chains, termination follows immediately.

In characteristic zero, such ranking functions are known and can be chosen \emph{intrinsic} (independent of auxiliary
choices) and even functorial with respect to smooth morphisms.  In the embedded setting, such
functions are built recursively from local, intrinsically defined data at a point (explained in detail in this section):
\begin{itemize}
\item the order (or multiplicity) $\mathrm{ord}_x(I)$ of the ideal (or of the defining equation);
\item auxiliary coefficient ideals on hypersurfaces of maximal contact, which exist in characteristic zero and allow an induction on dimension;
\item boundary data recording the exceptional divisor created so far (often via variants of $E$-order,
weighted orders, and related combinatorics).
\end{itemize}
Note that although their definition proceeds via auxiliary constructions (e.g. choosing a hypersurface of maximal contact to form coefficient ideals), the resulting invariant and the induced center are independent of these choices. These ingredients already appear in Hironaka's proof \cite{Hironaka1964} and they were
made explicit and canonical in forthcoming algorithmic constructions of Villamayor \cite{Villamayor1989},
Bierstone-Milman \cite{BierstoneMilman1997}, and W\l odarczyk \cite{Wlodarczyk2005}.

A recent refinement due to Abramovich-Temkin-W\l odarczyk gives a particularly transparent
functorial embedded resolution procedure via stack-theoretic \emph{weighted} blowings up
\cite{AbramovichTemkinWlodarczyk2024,Abramovich2018ICM}.  A striking feature is that the algorithm is driven by a
very simple invariant (and an associated maximal center) defined from an iterated maximal-contact
construction, without building the resolution history into the definition of the invariant. After a brief detour to review the necessary algebraic notions, we return to the description of this invariant in the next section.

\subsection{Algebraic preliminaries for the resolution invariants} In this section we briefly review the fundamental invariants historically used to measure singularities, and we illustrate the key differences in their behavior in characteristic zero and in characteristic $p$ through a toy example. Readers who find the discussion too abstract may wish to skip ahead to the next section, which contains an explicit worked-out example.
 
A \emph{marked ideal} on a smooth ambient variety $Y$ is a pair $(I,b)$ consisting of a coherent ideal
$I\subset\mathcal{O}_Y$ and an integer $b\in\mathbb{N}$ (the \emph{control}), with
\[
\mathrm{cosupp}(I,b)=\{p\in Y:\mathrm{ord}_p(I)\ge b\}.
\]
In characteristic $0$, when $(I,b)$ is of maximal order near $p$, one can choose a hypersurface
of maximal contact through $p$.  In local coordinates this may be encoded by a
\emph{maximal-contact element} $x\in\mathcal{O}_{Y,p}$, i.e.\ a regular parameter whose hypersurface
$V(x)$ contains the relevant equimultiple locus and is stable under admissible blow-ups.
Given such an element $x$, one associates to $(I,b)$ a \emph{coefficient ideal}
$\mathrm{Coeff}(I,b)$ on $V(x)$: informally, it is the ideal generated by the coefficients of an
$x$-expansion of generators of $I$, and equivalently it can be described using derivatives of $I$
up to order $b-1$ and then restricting to $V(x)$.
This coefficient ideal is the basic tool that allows one to descend in dimension.

When a simple normal crossings \emph{boundary divisor} $D$ is present on $Y$ (encoding the accumulated
exceptional divisor), the coefficient-ideal construction naturally splits into an exceptional monomial
part forced by~$D$ and a residual part measuring the genuinely new singularity.
Concretely, near a point $p$ choose local parameters
\[
u_1,\dots,u_r, v_{r+1},\dots,v_n
\]
such that $D=\sum_{i=1}^r m_i E_i$ with $E_i=\{u_i=0\}$; equivalently, the \emph{boundary ideal}
$I_Y(D)\subset\mathcal{O}_{Y,p}$ is invertible and generated by the monomial
\[
u^m := \prod_{i=1}^r u_i^{m_i}.
\]
If $V$ is a (weak) maximal-contact hypersurface through $p$ chosen transversal to $D$, then
$D\cap V$ is again SNC on $V$ and the restricted boundary ideal
\begin{equation}\label{factorization}
I_V(D\cap V)=I_Y(D)\cdot\mathcal{O}_{V,p}
\end{equation}
is invertible and generated by the same monomial viewed in $\mathcal{O}_{V,p}$.

Now form the coefficient ideal $\mathrm{Coeff}_V(I,b)$ on $V$.  Because the boundary components
contribute deterministically under controlled transforms, $\mathrm{Coeff}_V(I,b)$ typically contains a
largest monomial factor supported on $D\cap V$.  More precisely, there is a \emph{unique} monomial
$M$ in the boundary parameters (a product of powers of the $u_i|_V$) of maximal degree dividing
$\mathrm{Coeff}_V(I,b)$; factoring it off gives a canonical decomposition
\begin{equation}\label{eq:coeff-factor-elab}
\mathrm{Coeff}_V(I,b)= I_V(D\cap V)\cdot I^-,
\end{equation}
where $I_V(D\cap V)$ is this exceptional monomial (equivalently, the invertible ideal determined by the
boundary) and $I^-$ is the \emph{residual factor}, characterized by the property that it is not divisible by any further
monomial supported on $D\cap V$, see the toy kangoroo example in Section \ref{subsec:kangoroo}. Because exceptional components contribute monomial factors in a predictable way, one canonically
\emph{divides out} the maximal boundary-supported monomial factor $I_Y(D)$ from $\mathrm{Coeff}_V(I,b)$.

The \emph{residual order} (or \emph{shade}) is $\mathrm{ord}_p(I^-)$; it is the natural secondary term one
would like to use in a lexicographic resolution invariant.  In characteristic $0$ it behaves well after
factoring off the exceptional monomial, while in positive characteristic it can jump upward along
equiconstant branches (the kangaroo phenomenon, see below).

\subsection{A toy kangaroo example}\label{subsec:kangoroo}
We now illustrate these abstract concepts through an example from \cite{Hauser2010}, focusing on key differences between the characteristic zero and positive characteristic cases. Let $k$ be a field of characteristic $p=2$, and set $Y=\mathbb{A}_k^3=\mathrm{Spec}\,k[x,y,z]$. 
Consider the marked principal ideal $(I,b)$ with
\[
I=(f_0),\qquad b=2,\qquad f_0:=x^2+y^7+yz^4\in k[x,y,z].
\]
By definition at the origin 
\[
\mathrm{ord}_0(I)=\max\{n\ge 0:\ f_0\in \mathfrak m^n\}.
\]
where $\mathfrak{m}=(x,y,z)$ is the maximal ideal at the origin. 
Since $f_0\in\mathfrak m^2\setminus \mathfrak m^3$, we have $\mathrm{ord}_0(I)=2$.
For a principal ideal in a smooth ambient space, the cosupport of $(I,2)$ is the
multiplicity $\ge 2$ locus; for hypersurfaces this coincides with the singular locus and is detected by
first partial derivatives:
\[
\mathrm{cosupp}(I,2)=\{\,p\in Y:\ f_0(p)=0\ \text{and}\ \partial_x f_0(p)=\partial_y f_0(p)=\partial_z f_0(p)=0\,\}.
\]
In characteristic $2$ we compute
\[
\partial_x f_0 = 2x = 0,\qquad
\partial_y f_0 = 7y^6 + z^4 = y^6+z^4,\qquad
\partial_z f_0 = 4yz^3 = 0,
\]
hence near $0$ the cosupport is cut out by
\[
f_0=0,\qquad y^6+z^4=0.
\]
Note that in characteristic $0$ $\partial_x(x^2)=2x$ is not identically zero, so the Jacobian
conditions constrain $x$ as well; this is one manifestation of the purely inseparable pathology. Now take the regular parameter $x\in \mathfrak m$ and the hypersurface
\[
V:=V(x)=\{x=0\}\subset Y.
\]
We also have
\[
\mathrm{cosupp}(I,2)\subset V(x).
\]
Indeed, in characteristic $2$ we can write
\[
f_0=x^2+y^7+yz^4 \;=\; x^2 + y\,(y^6+z^4),
\]
hence on the locus defined by $y^6+z^4=0$, the additional condition $f_0=0$ forces $x^2=0$, and therefore
$x=0$ (already set-theoretically, and in fact after taking radicals).
Write the $x$-expansion
\[
f_0=x^2+a_1(y,z)x+a_0(y,z),\qquad a_1=0,\qquad a_0=y^7+yz^4.
\]
For $b=2$, the coefficient ideal on $V$ is generated by the coefficients of degree $<2$ in the
$x$-expansion (equivalently, by derivatives up to order $b-1=1$ and restriction to $V$), hence
\[
\mathrm{Coeff}_V(I,2)=\bigl(a_0,\ a_1\bigr)=(y^7+yz^4)\subset k[y,z]\cong \mathcal{O}_{V,0}.
\]
Its order at $0$ is $\mathrm{ord}_0(\mathrm{Coeff}_V(I,2))=\mathrm{ord}_0(y^7+yz^4)=5$.

We now follow a specific sequence of point blow-ups (in suitable affine charts) which produces an
\emph{antelope point} and then a \emph{kangaroo point}.  Denote by $f_i$ the corresponding strict transforms
(after the indicated coordinate substitutions).

\paragraph{First blow-up:}
In the affine chart $\{y \neq 0\}$ the blow-up at the origin simply means the coordinate change
\[
(x,y,z)\longmapsto (xy,\ y,\ zy).
\]
Then the proper transform of $f_0$ is 
\[
f_1 \;=\; x^2 + y^3\bigl(y^2+z^4\bigr).
\]
For $b=2$, the cosupport can again be described by the vanishing of $f_1$ and its first partials.  In
characteristic $2$,
\[
\partial_x f_1=0,\qquad
\partial_y f_1=\partial_y(y^5+y^3z^4)=y^4+y^2z^4=y^2(y^2+z^4),\qquad
\partial_z f_1=\partial_z(y^3z^4)=0,
\]
so locally $\mathrm{cosupp}\bigl((f_1),2\bigr)$ is cut out by
$f_1=0, y^2(y^2+z^4)=0$.
\paragraph{Second blow-up:} We blow-up at the origin and take the $\{z\neq 0\}$ affine chart where the coordinates transform as
\[
(x,y,z)\longmapsto (xz,\ yz,\ z).
\]
Then the proper transform of $f_1$ is
\[
f_2 = x^2 + y^3z^3\bigl(y^2+z^2\bigr).
\]
At this point the accumulated exceptional divisor has SNC components $\{y=0\}$ and $\{z=0\}$ with
multiplicities $3$ and $3$, so the boundary divisor $D$ has boundary ideal
\[
I_Y(D)=(y^3z^3)\subset \mathcal{O}_{Y,a_2}.
\]
Choosing $V=\{x=0\}$ transversal to $D$, we have $D\cap V$ SNC on $V$ and
\[
I_V(D\cap V)=I_Y(D)\cdot \mathcal{O}_{V,a_2}=(y^3z^3)\subset \mathcal{O}_{V,a_2}.
\]
Since $f_2=x^2+a_0$ with $a_0=y^3z^3(y^2+z^2)$, the coefficient ideal on $V$ is
\[
\mathrm{Coeff}_V\bigl((f_2),2\bigr)=(y^3z^3(y^2+z^2)).
\]
The maximal monomial factor supported on $D\cap V$ dividing this ideal is $y^3z^3$, hence the canonical
decomposition is
\[
\mathrm{Coeff}_V\bigl((f_2),2\bigr)
= I_V(D\cap V)\cdot I^-,
\qquad I^-=(y^2+z^2).
\]
The \emph{shade} (residual order) at the antelope point is
\[
\mathrm{shade}(a_2)=\mathrm{ord}_{a_2}(I^-)=\mathrm{ord}_0(y^2+z^2)=2.
\]

\paragraph{Third blow-up:}
Blow up the antelope point $a_2=(0,0,0)\in Y$. In the $z$-chart of this point blow-up we use the standard
substitution
\[
(x,y,z)\longmapsto (xz,\ yz,\ z).
\]
giving us the proper transform 
\[
f_3^{\mathrm{raw}}
= x^2+z^6\,y^3(y^2+1)
= x^2+z^6\,(y^5+y^3),
\]
where we use $\mathrm{char}(k)=2$.

To obtain the \emph{translational} kangaroo point, we move along the new exceptional divisor
$\{z=0\}$ to the point with $y=1$ in this chart, and recenter coordinates there by the translation
\[
y\longmapsto y+1,
\]
so that the chosen point becomes the origin.  In these recentered coordinates we get
\[
f_3^{\mathrm{raw}}
\;=\; x^2+z^6\,(y+1)^3\bigl((y+1)^2+1\bigr)
\;=\; x^2+z^6\,(y^5+y^4+y^3+y^2).
\]
At this point one performs the standard coordinate correction reflecting the failure of stable maximal
contact in characteristic $p$:
\[
x\longmapsto x+yz^3.
\]
Since $(x+yz^3)^2=x^2+y^2z^6$ in characteristic $2$, this eliminates the $y^2z^6$ term and the local
equation takes the announced form
\[
f_3 \;=\; x^2 + z^6\,(y^5+y^4+y^3).
\]
Here the relevant boundary is supported on the new exceptional component $\{z=0\}$ with multiplicity $6$,
so
\[
I_V(D\cap V)=(z^6).
\]
and the residual factor is
\[
I^-=(y^5+y^4+y^3)=y^3(y^2+y+1),
\]
and therefore
\[
\mathrm{shade}(a_3)=\mathrm{ord}_{a_3}(I^-)=\mathrm{ord}_0(y^5+y^4+y^3)=3.
\]
Thus, while the primary invariant $\mathrm{ord}(I)=2$ remains constant along this move, the secondary
invariant \emph{increases}:
\[
\mathrm{shade}(a_2)=2\quad\longrightarrow\quad \mathrm{shade}(a_3)=3.
\]
This upward jump of the residual order after factoring out the maximal boundary-supported monomial is the
\emph{kangaroo phenomenon}. The computations above highlight two characteristic-$p$ pathologies which do not occur in characteristic $0$:
\begin{itemize}
\item Ordinary derivatives may annihilate $p$th powers (here $\partial_x(x^2)=0$ in characteristic $2$),
      so the Jacobian conditions and differential descent behave differently; in characteristic $0$ one has
      $\partial_x(x^2)=2x\neq 0$ and the top locus interacts more rigidly with hypersurfaces of (weak) maximal contact.
\item After factoring off the exceptional monomial part of the coefficient ideal, the residual order
      (shade) is designed to decrease in characteristic $0$, but in characteristic $p$ it can jump
      upward at special points (kangaroo points), forcing additional structure or modified invariants in any
      termination strategy.
\end{itemize}

We finish this section with briefly recalling the ranking function which introduced in \cite{AbramovichTemkinWlodarczyk2024} in characteristic zero. 
Fix a point $p\in Y$ and choose a (partial) regular system of parameters
$x_1,\dots,x_k\in\mathcal{O}_{Y,p}$.
Following \cite{AbramovichTemkinWlodarczyk2024}, such a sequence is a \emph{maximal-contact sequence}
for $I$ at $p$ if it is constructed recursively by setting
\[
I_1:=I,\qquad b_i:=\mathrm{ord}_p\bigl(I_i\bigr),
\]
and requiring that for each $i\ge 1$ the parameter $x_i$ is a maximal-contact element for the marked
ideal $(I_i,b_i)$ at $p$.  One then defines the next ideal by restricting the coefficient ideal along
the successive intersections:
\[
I_{i+1}:=\mathrm{Coeff}(I_i,b_i)\big|_{V(x_1,\dots,x_i)}.
\]
From the resulting integers $b_i$ they define the invariant
\[
\mathrm{inv}_I(p):=(a_1,\dots,a_k),\qquad
a_i=\frac{b_i}{\prod_{j=1}^{i-1} b_j!},
\]
together with the associated (rationally generated) center
\[
J_p(I)=\bigl(x_1^{a_1},\dots,x_k^{a_k}\bigr).
\]
They prove that $\mathrm{inv}_I(p)$ and $J_p(I)$ are independent of the chosen maximal-contact
sequence, and that blowing up the maximal admissible center strictly decreases $\mathrm{inv}_I$
for the weak transform, yielding termination; see \cite{AbramovichTemkinWlodarczyk2024}.

\medskip
\noindent\emph{Toy example.}
Let $Y=\mathbb{A}^2=\mathrm{Spec}\,k[x,y]$ and let $I=(x^2,y^3)$ at the origin $p=(0,0)$.
Take the maximal-contact sequence $(x_1,x_2)=(x,y)$. Then
\[
b_1=\mathrm{ord}_p(I)=2,\qquad
I_2=\mathrm{Coeff}(I,2)\big|_{x=0}=(y^3),\qquad
b_2=\mathrm{ord}_p\bigl(I_2\bigr)=3,
\]
so
\[
\mathrm{inv}_I(p)=(a_1,a_2)=\Bigl(2,\frac{3}{2!}\Bigr)=\Bigl(2,\frac{3}{2}\Bigr),
\qquad
J_p(I)=\bigl(x^{2},\,y^{3/2}\bigr).
\]
Clearing denominators (e.g.\ with weights $(w_1,w_2)=(1,2)$) identifies the corresponding weighted
center with the integral ideal $(x^2,y^3)$, and the algorithm performs the associated weighted
blowing up (in the stack-theoretic sense of \cite{AbramovichTemkinWlodarczyk2024}).  By construction,
the next step replaces $I$ by its weak transform and strictly lowers $\mathrm{inv}_I$, and the process
iterates until the transform becomes monomial in suitable coordinates.

\noindent \textbf{Our strategy: drive to the monomial phase.}
Despite these obstructions, resolution by blow-ups in positive characteristics is known in low dimension: for excellent surfaces
 \cite{Lipman1978,CJS2009}, and for threefolds in
arbitrary characteristic in the work of Cossart-Piltant \cite{CossartPiltant2008,CossartPiltant2009,CossartPiltant2014}.

Rather than enforcing strict step-by-step monotonicity of a resolution invariant, we aim to push the
process toward the monomial phase, where the remaining complexity is toric geometry and hence
purely combinatorial.
We say we are in the \emph{monomial phase} at a point if
the residual factor is trivial:
\[
I^- = 1
\qquad\Longleftrightarrow\qquad
\mathrm{Coeff}_V(I,b)= I_V(D\cap V).
\]
Equivalently, the shade/residual order has reached its minimal value $0$.
In this regime, the remaining steps are governed only by the boundary stratification and can be carried
out by explicit toroidal (toric/combinatorial) transformations.

\smallskip
\noindent\emph{Toy example (monomial phase that is not simple normal crossings).}
Let $V=\mathbb{A}^2=\mathrm{Spec}(k[x,y])$ over a field of characteristic $p>0$, and consider the
boundary divisor
\[
D = 2E_x + 3E_y,\qquad E_x=\{x=0\},\ \ E_y=\{y=0\},
\]
with boundary ideal $I_V(D)=(x^2y^3)$.  This is already monomial phase (no residual factor),
but it is not reduced (hence not SNC in the strict sense).  The remaining behavior is combinatorial:
it is encoded by the exponent vector $(2,3)$ along the boundary strata.  For instance, viewing this as
a marked ideal with control $c=\mathrm{ord}_0(x^2y^3)=5$, a single blow-up of the codimension-two stratum
$Z=E_x\cap E_y$ yields controlled transforms of order $<c$ in each chart, illustrating the purely toroidal
descent mechanism from the monomial phase.

\medskip
Our early experiments attempted to enforce strict lexicographic descent at \emph{every} blow-up step.
This proved too ambitious: even in the toy simulator leading components can plateau, and a strict stepwise
score tends to induce fragile overfit conditionals. So we moved on to use a \emph{bounded-delay} condition: along the simulated canonical trajectory, the rank must improve at
least once in every window of fixed length $m$ (here $m=5$), until the monomial phase is detected.

\begin{definition}[Bounded-delay descent with window $m$]
\label{def:bounded-delay}
Let $(s_k)_{k\ge 0}$ be a trajectory of states and let
$R\colon \{\text{states}\}\to\mathbb{N}^d$ be a lexicographically ordered rank.
Write $R_k:=R(s_k)$ and define the running best $B_k = \min\nolimits_{\mathrm{lex}}\{R_0, R_1, \dots, R_k\}$.
We say that $R$ satisfies \emph{bounded-delay descent with window $m$} if, for every $k$ before termination,
\[
B_{k+m} < B_k,
\]
equivalently: within every block of $m$ consecutive steps there exists at least one strict improvement of
the best-so-far rank.
\end{definition}

This is conceptually a delayed Lyapunov-type condition.  To promote it to a termination statement one must
also ensure that $R$ takes values in a well-founded set (e.g.\ $\mathbb{N}^d$ with lex order), motivating
the discretization map~$\Pi$ in Section~\ref{subsec:discretize}.

\section{A toy canonical blow-up simulator}
\label{sec:setup}

In order to run controlled experiments on termination invariants, we fix a simplified setting in which
(i) singularities are hypersurface singularities and represented by finite monomial data 
and (ii) there is a deterministic canonical blow-up step that updates both the monomial data and an
intrinsic boundary (exceptional divisor) record.
The simulator is not intended to be a geometrically faithful implementation of resolution of singularities.
Its purpose is to provide a stable and nontrivial testbed in which known difficulties in positive
characteristic, notably long plateaux of naive order invariants and the need for boundary-aware secondary
invariants, can be studied systematically.

Note that our simulator follows a fixed, deterministic blow-up policy. This policy is implemented inside the
evaluation harness (outside the evolve block) and is therefore held constant across runs: AlphaEvolve
only mutates the ranking code, while the simulated blow-up dynamics do not evolve \cite{NovikovEtAl2025AlphaEvolve}.
In future iterations of these experiments, the blow-up policy itself could be moved into the evolve
block and learned jointly with the ranking function, potentially opening qualitatively new behaviors
and stronger empirical evidence for termination mechanisms \cite{NovikovEtAl2025AlphaEvolve}.

\subsection{Singularity selection and encoding}
Fix an integer $d\ge 2$ and affine coordinates $(x_1,\dots,x_{d-1},z)$ on $\mathbb{A}^d$
where $z$ is distinguished as an \emph{elimination variable} and $x_1,\dots,x_{d-1}$ are \emph{base variables}.
A \emph{monomial} is encoded by an exponent vector
$\mathbf{e}=(e_1,\dots,e_{d-1},e_z)\in \mathbb{N}^d$ via
\[
m(\mathbf{e}) = x_1^{e_1}\cdots x_{d-1}^{e_{d-1}} z^{e_z}.
\]
We write $|\mathbf{e}|:=\sum_v e_v$ for the total degree. In the toy simulator we restrict attention to hypersurface singularities presented (formally)
as
\begin{equation}\label{hypersurface}
f = z^p + \sum_{i=2}^N c_i\, m(\mathbf{e}_i),
\end{equation}
where $\mathrm{char}(k)=p>0$ and the leading term is \emph{monic purely inseparable} (so the coefficient of
$z^p$ is $1$).  Our encoding discards the coefficients $c_i$ and retains only a chosen list of exponent data,
i.e.\ a coarse Newton-combinatorial shadow of the hypersurface.

Each retained monomial also carries a symbolic \emph{tag} $t$.  Formally, $t$ is just a label (a string-valued type such as \texttt{pure-$z$}, \texttt{mixed}, \texttt{oblique}, \texttt{heavy-tail}, \dots),
used only by certain proxy features in Section~\ref{sec:features} to implement family-dependent heuristics. Tags are labels describing the intended role/type of a monomial in the initial synthetic benchmark, kept fixed throughout the blow-up process.
The ranking functions studied in the main conjectures are required to be intrinsic in the sense that they do
\emph{not} depend on these tags; the tags are only a bookkeeping device for constructing and stress-testing
benchmark families. Then the hypersurface \eqref{hypersurface} will be encoded as an \emph{ideal specification}: it is a finite list
\[
\mathcal{I}=\{(t_i,\mathbf{e}_i)\}_{i=1}^N,
\]
where $t_i$ is the tag and $\mathbf{e}_i\in\mathbb{N}^d$ is the exponent vector.
We always include the distinguished purely inseparable leading term $z^p$ by requiring that
\[
(t_1,\mathbf{e}_1)=(\texttt{pure-$z$},(0,\dots,0,p)).
\] 
A \emph{state} in our simulator is a pair
$
S=(\mathcal{I},\partial),
$
where $\mathcal{I}$ is an ideal specification and
\[
\partial:\{x_1,\dots,x_{d-1},z\}\to \mathbb{N}
\]
is a \emph{boundary function}.  This boundary function encodes, in the simplest possible way, the
(proper transform of) exceptional divisors from previous blow-ups: locally at the point $p$ the exceptional divisor is 
\[
E = \sum_{v\in \{x_1,\ldots x_{d-1},z\}} \partial(v)\,E_v,
\qquad E_v=\{v=0\},
\]
and its ideal is 
\[
I_Y(E) = \Bigl(\prod\nolimits_v v^{\partial(v)}\Bigr) \subset \mathcal{O}_{Y,p}.
\]
which is exactly the exceptional monomial factor in one may factor off from coefficient ideals and transforms
when defining boundary-adjusted invariants (e.g.\ $E$-order, $w$-$\mathrm{ord}$, residual order/shade);
see the discussion in \cite{Hauser2003,Hauser2008}, and our discussion in Section \ref{subsec:kangoroo}.
In our simulator, $\partial$ is the only persistent memory carried between steps.  Intrinsic ranking
functions are required to be history-free, but several boundary-aware \emph{proxy} invariants (such as
$E$-order and boundary-adjusted order proxies) are allowed to depend on~$\partial$.

\smallskip
\noindent\emph{Toy example (encoding a hypersurface and its boundary).}
Take $d=3$ with coordinates $(x,y,z)$ and $\mathrm{char}(k)=p$.
Consider the hypersurface
\[
f = z^p + x^3y^2 + x^8
\]
near the origin.  The corresponding ideal specification (forgetting coefficients) is
\[
\mathcal{I}=\bigl\{(\texttt{pure-$z$},(0,0,p)),\ (\texttt{mixed},(3,2,0)),\ (\texttt{pure-$x$},(8,0,0))\bigr\}.
\]
If, in addition, the current exceptional divisor has (local) equation $x^2y=0$, i.e.\
$E=2E_x+1E_y$, then the boundary function is
\[
\partial(x)=2,\qquad \partial(y)=1,\qquad \partial(z)=0,
\]
and the state is $S=(\mathcal{I},\partial)$.  The associated boundary ideal is $I_{\mathbb{A}^3}(E)=(x^2y)$.

Define the \emph{order proxy} of $\mathcal{I}$ as minimum of the degrees of monomials in the monic hypersurface:
\[
\mathrm{ord}(\mathcal{I}) := \min_{1\le i\le N} |\mathbf{e}_i|.
\]
If $\mathcal{I}$ contains a \emph{pure} $z$-power, i.e.\ some $\mathbf{e}_i$ supported only on $z$, define
\[
\mathrm{ord}_z(\mathcal{I}):=\min\{\,e_{i,z}:\ (t_i,\mathbf{e}_i)\in\mathcal{I},\ e_{i,x_j}=0\ \forall j\,\}.
\]
At each step we use the \emph{exceptional exponent}
\begin{equation}\label{exc}
\mathrm{exc}(\mathcal{I}):=
\begin{cases}
\mathrm{ord}_z(\mathcal{I}), & \text{if a pure }z\text{-power occurs in }\mathcal{I},\\
\mathrm{ord}(\mathcal{I}),   & \text{otherwise}.
\end{cases}
\end{equation}

\subsection{Canonical center selection and blow-up}

Our simulator follows a fixed deterministic blow-up procedure (outside the evolve block). 
Given a state $S=(\mathcal I,\partial)$, we choose the next center $C(S)$ as follows.
Let $z$ be the elimination variable.
\begin{enumerate}[leftmargin=*]
\item If there exists a pure base monomial $x_j^{a}$ (i.e.\ an exponent vector supported on a single base variable), 
      choose such a $j$ with maximal exponent $a$, and set
      \[
      C(S)=V(x_j,z)\subset \mathbb{A}^d.
      \]
\item Otherwise, consider all base monomials in $\mathcal I$ (those with $e_z=0$). Choose one of minimal total degree 
      (ties are broken deterministically by the fixed input ordering), then choose a base variable $x_j$ appearing with maximal exponent 
      in that monomial, and set $C(S)=V(x_j,z)$.
\item If no base monomials exist (i.e.\ every monomial involves $z$), we use a \emph{divisor-type fallback} and set
      \[
      C(S)=V(z).
      \]
      In the implementation this should be understood as a deterministic \emph{normalization step in the $z$-direction} (a ``$z$-chart'' rule)
      rather than a geometrically meaningful blow-up of the ambient space along a Cartier divisor.
\end{enumerate}

The motivation for this center selection is the following. 
In characteristic-zero algorithms (e.g.\ Villamayor, Bierstone--Milman, W{\l}odarczyk), the next center is chosen canonically
as the smooth maximal stratum of a well-founded local invariant, with normal crossings conditions relative to the accumulated exceptional divisor.
Our center-selection rule is a deliberately crude Newton-combinatorial proxy for this philosophy: we treat $z$ as an elimination direction and 
typically force repeated interaction between $z$ and a single base direction by blowing up the codimension-$2$ coordinate stratum $V(x_j,z)$.

The fallback $C(S)=V(z)$ occurs only when no base monomials are present. 
Geometrically, blowing up an effective Cartier divisor is an isomorphism of the ambient space; however, it can still induce a nontrivial
transformation of the marked ideal/boundary data used by resolution invariants. 
This principle is standard in the marked-ideal formalism and is stated explicitly, for example, in W{\l}odarczyk's discussion of codimension-one
components: the blow-up is an isomorphism, but the marked ideal transforms nontrivially. 
In our toy model, the branch $C(S)=V(z)$ plays the role of such a divisor-type normalization step: the simulator continues to update the monomial
exponent data via the same deterministic rewrite rule, but we do not interpret this branch as a genuine birational modification of the ambient space.

Once the center is fixed, the simulator updates the state $S=(\mathcal I,\partial)$ and produces a new state
$S'=(\mathcal I',\partial')$ as follows. Set $\mathrm{exc}=\mathrm{exc}(\mathcal I)$.

\begin{enumerate}[leftmargin=*]
\item \textbf{Boundary update.}
\begin{itemize}[leftmargin=*]
\item If $C(S)=V(x_j,z)$, we first restrict the boundary to the variables meeting the center,
\[
\partial^{\mathrm{res}}(u)=
\begin{cases}
\partial(u), & u\in\{x_j,z\},\\
0, & \text{otherwise},
\end{cases}
\]
and then add the new exceptional contribution supported on $x_j$:
\[
\partial'(u)=\partial^{\mathrm{res}}(u)\ (u\neq x_j),
\qquad
\partial'(x_j)=\partial^{\mathrm{res}}(x_j)+\mathrm{exc}.
\]
\item If $C(S)=V(z)$, we restrict to the $z$-component
\[
\partial^{\mathrm{res}}(u)=
\begin{cases}
\partial(u), & u=z,\\
0, & \text{otherwise},
\end{cases}
\]
and then (in the implementation) we also add the exceptional contribution on $z$:
\[
\partial'(z)=\partial^{\mathrm{res}}(z)+\mathrm{exc},
\qquad
\partial'(u)=0\ \ (u\neq z).
\]
\end{itemize}

\item \textbf{Monomial exponent rewrite.}
Let $C(S)=V(x_j,z)$ or $C(S)=V(z)$, and let $v$ denote the corresponding chart variable, i.e.
\[
v=
\begin{cases}
x_j, & \text{if } C(S)=V(x_j,z),\\
z, & \text{if } C(S)=V(z).
\end{cases}
\]
For each tagged monomial $(t,\mathbf{e})\in\mathcal I$ with exponent vector $\mathbf{e}=(e_u)_u$, define a new exponent vector $\mathbf{e}'$
by the following deterministic rewrite:
\begin{itemize}[leftmargin=*]
\item (\emph{$z$-mixing into the $v$-chart}) If $e_z>0$, replace $e_v$ by $e_v+e_z$ (leaving all other coordinates unchanged).
      In particular, in the fallback case $v=z$ this replaces $e_z$ by $2e_z$.
\item (\emph{Divide out the exceptional factor}) Replace $e_v$ by $\max\{0,e_v-\mathrm{exc}\}$.
\item (\emph{Clean-up}) Remove all coordinates whose exponent becomes $0$; if all exponents become $0$ the monomial is discarded.
\end{itemize}
The resulting list of tagged monomials is $\mathcal I'$.
\end{enumerate}

Tags are propagated \emph{unchanged}: if $(t,\mathbf{e})\in\mathcal I$ then the transformed monomial in $\mathcal I'$ is recorded as $(t,\mathbf{e}')$.
Thus tags should be viewed as persistent labels describing the intended role/type of a monomial in the initial synthetic benchmark,
not as dynamically recomputed classifications of the transformed exponent vector.

\noindent \textbf{Iteration and stopping at monomial phase.}
Starting from an initial state $S_0=(\mathcal I_0,\partial_0)$ with $\partial_0\equiv 0$, we define a sequence
of states by repeatedly applying the canonical step
\[
S_{k+1} := \bigl(\mathcal I_{k+1},\partial_{k+1}\bigr)
\qquad\text{from}\qquad
S_k=(\mathcal I_k,\partial_k),
\]
using the center selection rule, boundary update, and monomial transform described above.
We stop either after a fixed step cap, or when a \emph{monomial phase} condition is met:
no monomial in $\mathcal I_k$ involves $z$, and all tags belong to a designated set of monomial tags
(as implemented).
Conceptually, this models reaching a regime where the remaining singularity is combinatorially monomial and is
expected to be tractable by purely toroidal methods.

\section{Feature selection} 
\label{sec:features}

Recall from Section \ref{sec:setup} that we built a test set of monic hypersurface singularities in the  affine ambient space
$\mathbb{A}^d=\mathrm{Spec} k[x_1,\dots,x_{d-1},z],
$
in characteristic $p=\mathrm{char}(k)$, with $z$ distinguished as an elimination variable.
Such a hypersurface is encoded by an ideal specification $\mathcal I=\{(t_i,\mathbf{e}_i)\}_{i=1}^N$ which is a finite tagged list 
$\mathbf{e}_i=(e_{i,1},\dots,e_{i,d-1},e_{i,z})\in\mathbb{N}^d$ and $t_i$ is a symbolic label ("pure power", "mixed",
"oblique", etc.). A state carries a boundary function $
\partial:\{x_1,\dots,x_{d-1},z\}\to \mathbb{N}$,
interpreted as multiplicities of exceptional components "$v=0$" through the current (simulated) point.
Under the canonical step with center $V(v,z)$, the boundary is (i) restricted to components meeting the center
and (ii) augmented by a new contribution supported on $v$ of weight $\mathrm{exc}(\mathcal I)$.

The blow-up simulator we built in  Section \ref{sec:setup} selects canonical centers and applies blow-ups at those: it updates $\mathcal{I}$ and the exceptional divisor data encoded by the boundary map $\partial$.
Most intrinstic invariants we feed to the evolving ranking function are \emph{proxies}:
they are combinatorial functions of exponent vectors (and the boundary) designed to correlate with familiar
geometric quantities used in resolution algorithms (order, $w$-ord, $E$-order, directrix dimension $\tau$,
Jacobian-type signals, Frobenius phenomena, etc.).
The goal of proxy design is pragmatic:
\begin{itemize}
\item the proxies should be intrinsic to the state $S=(\mathcal I,\partial)$ (no history dependence);
\item they should detect and penalize known failure modes in characteristic $p$ (long plateaux, Frobenius-flatness,
      sensitivity to tie-breaking, delayed shade effects);
\item they must support the evaluation harness constraint that monomial phase forces the primary rank
      component to vanish, and that certain alignment drops (notably in the order proxy) should be reflected by rank descent.
\item Boundary-aware proxies (notably $E$-order and $w$-ord surrogates) depend on $\partial$, and $\partial$ provides
a minimal intrinsic mechanism for remembering exceptional complexity without violating the purity requirement
on ranking functions.
\end{itemize}

\subsection{Feature list}\label{subsec:features} In this section we list our invariants which are used as features of the evolved ranking function.  
Fix a state $S=(\mathcal I,\partial)$ with elimination variable $z$.
Let
\[
\mathcal M=\{\mathbf{e}_i\}_{i=1}^N\subset\mathbb{N}^d
\qquad\text{and}\qquad
\mathcal{M}_{\mathrm{base}}=\{\mathbf{e}\in\mathcal M: e_z=0\}
\]
be the full and base exponent sets. For $\mathbf{e}\in\mathbb{N}^d$ write $|\mathbf{e}|=\sum_v e_v$.
Let $\mathrm{exc}(\mathcal I)$ be as in Section~\ref{sec:setup}.
We also use the \emph{initial set}
\[
\mathcal{M}_{\min}:=\{\,\mathbf{e}\in\mathcal{M}:\ |\mathbf{e}|=\mathrm{exc}(\mathcal I)\,\},
\]
i.e. the monomials of minimal total degree, except that in the monic regime $\mathrm{exc}(\mathcal{I})$ coincides with
the pure $z$-order by definition.
The features $f_0,\dots,f_{25}$ are listed in the following table, along with their Gemini-generated feature names and heuristic meaning. Some of them (e.g $f_{14}$) were suggested by the AlphaEvolve experiments and added to the initial feature set during the evolution of experiments. 

\begin{longtable}{@{}r p{0.28\textwidth} p{0.62\textwidth}@{}}
\toprule
Index & Feature name & Definition (proxy) and purpose \\
\midrule
\endhead

$f_0$ &
\texttt{max\_order} &
\textbf{Order proxy.}
$f_0=\mathrm{exc}(\mathcal{I})$, see \eqref{exc}

\emph{Purpose:} a primary descent target analogous to order/multiplicity. The harness imposes strong alignment
constraints when $f_0$ drops. \\[0.3em]

$f_1$ &
\texttt{elimination\_order} &
\textbf{Base order proxy.}
If $\mathcal M_{\mathrm{base}}\neq\varnothing$, set $f_1=\min_{\mathbf{e}\in\mathcal M_{\mathrm{base}}}|\mathbf{e}|$; otherwise $f_1=0$.

\emph{Purpose:} measures the smallest "coefficient" term (in the base variables) competing with the elimination term;
plateaux often occur when $f_1$ is close to $f_0$. \\[0.3em]

$f_2$ &
\texttt{dim\_max\_locus\_proxy} &
\textbf{Max-locus dimension proxy.}
Let $T$ be the set of variables $v$ such that $e_v>0$ for some $\mathbf{e} \in\mathcal M_{\min}$.
$f_2$ is the number of variables not in $T$ (i.e. free directions).

\emph{Purpose:} a crude surrogate for the dimension of the locus of maximal order. \\[0.3em]

$f_3$ &
\texttt{comp\_max\_locus\_proxy} &
\textbf{Max-locus component count proxy.}
$f_3=|\mathcal M_{\min}|$.

\emph{Purpose:} distinguishes single-branch vs.\ multi-branch initial data and stresses tie-breaking. \\[0.3em]

$f_4$ &
\texttt{boundary\_count} &
\textbf{Exceptional component count.}
Set $f_4:=|\{v:\partial(v)>0\}|$.
\emph{Purpose:} measures boundary complexity at the current stage; used as a mild regularizer and to break ties. \\[0.3em]

$f_5$ &
\texttt{shade\_penalty} &
\textbf{Shade proxy (small-degree mixed residuals).}
Count monomials with total degree in the window
$p\le |e|<2p$ and whose tag is mixed/oblique.

\emph{Purpose:} flags low-degree residual terms that can create delayed progress or temporary regressions, mimicking
the role of "shade"-type secondary invariants in positive characteristic heuristics. \\[0.3em]

$f_6$ &
\texttt{jacobian\_vanish\_flag} &
\textbf{Jacobian-vanishing boolean (Frobenius-flatness).}
Return $1$ iff for every monomial $\mathbf{e}\in\mathcal M$ and every exponent $e_v>0$ one has $p\mid e_v$.
Equivalently, all formal first partial derivatives of the monomials vanish in characteristic $p$.

\emph{Purpose:} detects purely inseparable initial behavior, a primary source of characteristic-$p$ difficulty. \\[0.3em]

$f_7$ &
\texttt{newton\_slope} &
\textbf{Newton-slope proxy.}
If a pure $z^{f_0}$ monomial exists, return $f_0/f_1$ with the convention $f_0$ if $f_1\le 0$; otherwise $0$.

\emph{Purpose:} a rough measure of relative steepness of the elimination direction compared to the base order,
often correlating with imminent decreases in weighted-order surrogates under the toy transform. \\[0.3em]

$f_8$ &
\texttt{e\_order\_boundary\_proxy} &
\textbf{Boundary $E$-order proxy.}
For each $\mathbf{e}\in\mathcal M_{\min}$ compute
$
E(\mathbf{e}):=\sum_{v\in\{x_1,\dots,x_{d-1}\}}\min(e_v,\partial(v)),
$
and set $f_8:=\min_{\mathbf{e}\in\mathcal M_{\min}} E(\mathbf{e})$.

\emph{Purpose:} mimics the contribution of exceptional components to the maximal-order locus and supports
boundary-sensitive descent criteria. \\[0.3em]

$f_9$ &
\texttt{monomial\_phase} &
\textbf{Monomial-phase detector.}
Return $1$ iff no monomial involves $z$ and all tags lie in the designated monomial-like tag set; otherwise $0$.

\emph{Purpose:} the harness enforces the normalization $f_9=1\Rightarrow \mathrm{rank}[0]=0$; this encodes the goal
of reaching a monomial/toroidal regime. \\[0.3em]

$f_{10}$ &
\texttt{inseparable\_initial\_flag} &
\textbf{Initial-form inseparability.}
Return $1$ iff for every $\mathbf{e}\in\mathcal M_{\min}$ and every coordinate $v$ with $e_v>0$ one has $p\mid e_v$.

\emph{Purpose:} distinguishes inseparability at the \emph{initial level} from higher-order mixed terms; this fixes an
earlier bug where inseparability was tested on all monomials rather than on $\mathcal M_{\min}$. \\[0.3em]

$f_{11}$ &
\texttt{plateau\_risk} &
\textbf{Plateau-risk heuristic.}
If $f_1=0$, return $f_0$; otherwise return $1/(1+|f_0-f_1|)$.

\emph{Purpose:} flags the regime $f_0\approx f_1$, which in our benchmarks correlates with long constant-order runs. \\[0.3em]

$f_{12}$ &
\texttt{frobenius\_defect} &
\textbf{Frobenius defect count.}
Count variables $v$ that appear in some monomial but whose exponents are all divisible by $p$ across \emph{all} monomials.
\emph{Purpose:} identifies coordinates invisible to first derivatives and often responsible for delayed descent. \\[0.3em]

$f_{13}$ &
\texttt{center\_complexity} &
\textbf{Center-selection stress proxy.}
Among pure base monomials $x_j^{a}$ (single base variable, excluding $z$), return $\max a$, and return $0$ if none exist.

\emph{Purpose:} correlates with how strongly the canonical center selection (Section~\ref{sec:setup}) is biased toward a
single coordinate line. \\[0.3em]

$f_{14}$ &
\texttt{weighted\_order\_proxy} &
\textbf{Boundary-aware weighted-order proxy ($w$-ord surrogate).} Assume the $z$-monic term is $z^{f_0}$. 
For each monomial $\mathbf{e}\in\mathcal M$ \emph{except} the pure $z^{f_0}$ term, set $e_z:=e_z$ and let $e^{\mathrm{base}}$
be the restriction of $e$ to base variables.
Define a residualized base exponent vector by
$
e^{\mathrm{res}}_v:=\max(0,e_v-\partial(v))\qquad (v\neq z).
$
If $e_z>0$, record $|e^{\mathrm{res}}|/e_z$; if $e_z=0$, record $|e^{\mathrm{res}}|/f_0$.
Set $f_{14}$ to the minimum recorded value.

\emph{Purpose:} provides a secondary descent target in constant-order regimes. The explicit exclusion of the pure
$z^{f_0}$ term is essential in monic inputs; otherwise the minimum would be trivially $0$. \\[0.3em]

$f_{15}$ &
\texttt{tau\_directrix\_proxy} &
\textbf{Directrix/$\tau$-proxy (base).}
Let $\mathcal M_{\mathrm{base},\min}:=\{\mathbf{e}\in\mathcal M_{\mathrm{base}}:\ |\mathbf{e}|=f_1\}$.
Return the number of base variables $x_j$ such that $e_j=0$ for all $\mathbf{e}\in\mathcal M_{\mathrm{base},\min}$.

\emph{Purpose:} a crude proxy for the codimension of the directrix (or, dually, the dimension of the space of
"free" directions) on the base initial form. \\[0.3em]

$f_{16}$ &
\texttt{e\_order\_elim} &
\textbf{Elimination-exponent proxy.}
Return $\min\{e_z>0:\ \mathbf{e}\in\mathcal M\}$, with the convention $0$ if no monomial involves $z$.

\emph{Purpose:} distinguishes monic $z^p$ regimes from higher $p^2$ regimes (e.g.\ $z^9$). \\[0.3em]

$f_{17}$ &
\texttt{embedding\_dim\_proxy} &
\textbf{Embedding-dimension proxy.}
Return the number of variables that occur with positive exponent in at least one monomial.

\emph{Purpose:} mild structural signal for eliminating irrelevant variables; mainly used as a tie-breaker. \\[0.3em]

$f_{18}$ &
\texttt{wildness\_index} &
\textbf{Wildness proxy (mixed, non-$p$-divisible exponents).}
Count tagged mixed/oblique monomials with $|\mathbf{e}|\ge p$ for which some exponent is not divisible by $p$.

\emph{Purpose:} separates Frobenius-flat mixed behavior from genuinely "wild" mixed terms (a common source of
non-monotone secondary invariants). \\[0.3em]

$f_{19}$ &
\texttt{base\_dim\_max\_locus\_proxy} &
\textbf{Max-locus dimension proxy on the base.}
Repeat the construction of $f_2$ using only $\mathcal M_{\mathrm{base}}$ and only base variables.

\emph{Purpose:} focuses on coefficient/companion data when $z$-monicity masks geometry at the total level. \\[0.3em]

$f_{20}$ &
\texttt{base\_comp\_max\_locus\_proxy} &
\textbf{Component count proxy on the base.}
Return $|\mathcal M_{\mathrm{base},\min}|$.
\emph{Purpose:} distinguishes single vs.\ multiple minimal base constraints, important in cross/competition families. \\[0.3em]

$f_{21}$ &
\texttt{hilbert\_samuel\_base\_value} &
\textbf{Hilbert-Samuel proxy from the base initial monomial ideal.}
Let $I_{\mathrm{base}}\subset k[x_1,\dots,x_{d-1}]$ be the monomial ideal generated by $\mathcal M_{\mathrm{base},\min}$.
Set $n:=f_1+1$ and define
$
H(n):=\#\{\,\text{base monomials of total degree }<n\text{ not contained in }I_{\mathrm{base}}\,\}.
$
Return $f_{21}:=H(n)$.

\emph{Purpose:} a richer "size" signal than base order alone; sensitive to cross/competition and helps break
plateaux where $f_1$ is constant. \\[0.3em]

$f_{22}$ &
\texttt{jacobian\_min\_order} &
\textbf{Jacobian-order proxy (nonzero partials only).}
For each monomial exponent $\mathbf{e}\in\mathcal M$ and each variable $v$ with $e_v>0$ and $p\nmid e_v$, form the derivative
exponent vector $\mathbf{e}-\mathbf{1}_v$ (subtracting $1$ in the $v$-coordinate). Return the minimum of $|\mathbf{e}-\mathbf{1}_v|$.
If no such pair exists, return a large sentinel.

\emph{Purpose:} quantifies how "high" the first nonzero Jacobian information appears in characteristic $p$. \\[0.3em]

$f_{23}$ &
\texttt{jacobian\_nonzero\_partials} &
\textbf{Jacobian-density proxy.}
Count pairs $(\mathbf{e},v)$ with $\mathbf{e}\in\mathcal M$, $e_v>0$, and $p\nmid e_v$.

\emph{Purpose:} measures how many first partials are nonzero at the monomial level; low density correlates with
hard purely inseparable behavior. \\[0.3em]

$f_{24}$ &
\texttt{padic\_depth\_initial} &
\textbf{$p$-adic depth on the initial set.}
For each $\mathbf{e}\in\mathcal M_{\min}$ and each coordinate $v$ with $e_v>0$, compute $\nu_p(e_v)$ and return the minimum.

\emph{Purpose:} detects deeper $p$-power structure (e.g.\ $p^2$-type exponents) already present in the initial data. \\[0.3em]

$f_{25}$ &
\texttt{boundary\_mult\_sum} &
\textbf{Total boundary mass.}
 $f_{25}=\sum_v \partial(v)$.
\emph{Purpose:} captures cumulative exceptional complexity beyond the mere count $f_4$, and is used by evolved rankers to discourage spending too much boundary mass without progress in earlier components. \\

\bottomrule
\end{longtable}

We implemented a harness-imposed alignment:
the evaluation harness enforces an order alignment constraint. This means that whenever the canonical step strictly decreases
the primary order proxy $f_0$, the produced rank must strictly decrease at that step.
A weaker alignment penalty is also applied to decreases of the weighted-order proxy $f_{14}$.
These constraints were essential in practice to prevent the search from learning rankers that ignore genuine
progress signals while optimizing only for delayed or secondary effects.

\section{Evaluation harness and scoring system}

AlphaEvolve is tasked with evolving a pure function
\[
R:\mathbb{R}^{26} \longrightarrow \mathbb{R}^k,\qquad
R(\mathbf{f})=(r_0(\mathbf{f}),\dots,r_{k-1}(\mathbf{f})),
\]
where $\mathbf{f}\in\mathbb{R}^{25}$ denotes the feature vector extracted from a simulator state.
The final ranking function was set to be the lexicographic order on $\mathbb{R}^k$.

A hard boundary constraint is imposed using the monomial-phase indicator $f_9\in\{0,1\}$ (Section~\ref{sec:features}):
\begin{equation}\label{eq:monomial-normalization}
f_9=1 \ \Longrightarrow\ r_0(\mathbf{f})=0,
\qquad\qquad
f_9=0 \ \Longrightarrow\ r_0(\mathbf{f})>0.
\end{equation}
This enforces the normalization that the terminal regime is the monomial phase and it is detected by the leading rank component
vanishing exactly, while every non-monomial state carries strictly positive leading value.

Finally, we enforce \emph{purity}: the function $R$ must be deterministic and depend only on the input feature vector.
In particular, $R$ may not use global mutable state, randomization, wall-clock time, or any form of history-memory. To make this rigorous, we let the evaluation harness check determinism by evaluating $R$ twice on the same feature vector
and requiring identical outputs.

For each test ideal, we simulate a trajectory of states
\[
S_0 \to S_1 \to \cdots \to S_T,
\qquad T\le K,
\]
where $K$ is a fixed step cap (in the main experiments, $K=30$), and simulation stops early if monomial phase is reached.
Let $\mathbf{f}_t\in\mathbb{R}^d$ denote the feature vector extracted from $S_t$, and write
\[
R_t := R(\mathbf{f}_t)\in\mathbb{R}^k.
\]
All scoring rules below are evaluated on the finite sequence $(R_t)_{t=0}^T$ together with $(\mathbf{f}_t)_{t=0}^T$.

The harness computes a nonnegative \emph{violation count} for each test instance. A violation occurs whenever one of the following conditions fails:
\begin{enumerate}[leftmargin=2em]
\item 
The function $R$ must return a $k$-tuple of finite real numbers at every simulated step.
Crashes, NaNs/Infs, or non-tuples are treated as structural failures and incur a large fixed penalty.

\item 
Condition \eqref{eq:monomial-normalization} is enforced at every step $t$.

\item The bounded-delay ranking condition as described in Section \ref{sec:background}. Fix a window length $m\in\mathbb{N}$ (typically $m=5$ in the main runs).
Define the running best value (lexicographic minimum)
\[
B_t := \min\nolimits_{\mathrm{lex}}\{R_0,\dots,R_t\}.
\]
We require that before monomial phase, the sequence $(B_t)$ improves at least once in every block of length $m$:
equivalently, writing $\tau$ for the first index with $f_9(\mathbf{f}_\tau)=1$ (or $\tau=T+1$ if monomial phase is never reached),
we require that for every $t<\tau$ there exists $s\in\{t+1,\dots,\min(t+m,\tau)\}$ such that
\[
B_s <_{\mathrm{lex}} B_{s-1}.
\]
The harness tracks the index of the most recent strict improvement of $B_t$ and increments the violation count if the
gap reaches $m$ while still in non-monomial phase. 

\item Certain feature drops are treated as \emph{immediate progress signals} and are required to be reflected by an immediate rank decrease.
Specifically, if the primary order proxy $f_0$ satisfies $f_0(\mathbf{f}_t)<f_0(\mathbf{f}_{t-1})$, then we impose the penalty
\[
R_t \not<_{\mathrm{lex}} R_{t-1} \quad\Longrightarrow\quad \text{add a (heavier) violation increment.}
\]
Similarly, if the boundary-aware weighted-order proxy $f_{14}$ decreases, $f_{14}(\mathbf{f}_t)<f_{14}(\mathbf{f}_{t-1})$, then
failure of $R_t<_{\mathrm{lex}}R_{t-1}$ incurs a (lighter) penalty.
These alignment constraints are not asserted as mathematical truths about genuine resolution invariants;
they are a deliberate design choice to prevent the evolutionary search from ignoring clear proxy progress.

\end{enumerate}

\paragraph{Secondary diagnostics (tie-breakers).}
In addition to violations, the harness records:
\begin{itemize}[leftmargin=2em]
\item the number of indices $t$ with $R_t>_{\mathrm{lex}}R_{t-1}$ (local increases), and
\item the maximal length of a plateau, i.e.\ the largest number of consecutive indices with $R_t=R_{t-1}$.
\end{itemize}
These quantities are used as secondary objectives (tie-breakers) and as diagnostic outputs; they do not replace the feasibility
constraints above.

To stabilize search and reduce overfitting to a small subset of test instances, we evaluate the same candidate $R$
on a three-stage curriculum: an easy prefix of the test set, a medium prefix, and the full suite.
Each stage contributes additively to the primary violation total, with increasing weights.
The harness reports multiple objectives, including the weighted total violation count and a worst-case (max-per-instance) violation count,
as well as per-stage breakdowns.

Across the sequence of experiments described in Section~\ref{sec:how-run}, the scoring system was refined iteratively.
Early versions required strict lexicographic descent at every step and were therefore too rigid in characteristic-$p$ plateau regimes.
Later versions introduced the bounded-delay best-so-far descent criterion, the hard normalization \eqref{eq:monomial-normalization},
purity checks to rule out stateful cheats, and curriculum staging with explicit alignment penalties.
The final scoring used in Section \ref{sec:results} is the outcome of this iterative redesign.

\section{Evolution of experiments and dead ends}\label{section:experimentI}

In this section we summarize the main iterations of the experiment, the lessons we learned and improvements/adjustments we made along the way.

\subsection{Plateau budgets and the meaning of $m$}
Our earliest evaluation harness asked strict lexicographic decrease at every step.
This is the natural termination mechanism in characteristic zero, but in positive characteristic
it is too strong even for our toy simulator: many trajectories contain long stretches where the most
obvious proxies do not change, or even move in the
wrong direction for several steps before improving.
In such a regime, a strict stepwise requirement makes it nearly impossible for the evolutionary
search to make progress: any candidate ranker that tolerates a temporary plateau is immediately
discarded as non-terminating.

The bounded-delay criterion of Definition~\ref{def:bounded-delay} was introduced precisely to
make the search robust to this phenomenon.
Operationally, it replaces the requirement $\mathcal R(s_{k+1})<\mathcal R(s_k)$ for all $k$
by the weaker requirement "the best-so-far rank must improve at least once every $m$ steps."
In the focused $p=3$, $d=4$ benchmark the final evolved rankers exhibit very short plateaux:
\begin{itemize}[leftmargin=*]
\item for the continuous two-component ranker of Section~8, the worst plateau length is $2$;
\item for the discretized ranker of Section~\ref{subsec:discretize}, the worst plateau length is again $2$
(and in stages 1 and 2 it drops to $1$).
\end{itemize}
By contrast, on the earlier broad benchmark (Section \ref{sec:toobroad}) the best evolved candidate we
observed still exhibited plateaux as long as $28$ steps, producing many bounded-delay violations.

Choosing $m=5$ strikes a pragmatic balance in our toy setting: it is short enough to
exclude genuine cycles, but long enough to allow the kind of bounded kangaroo behavior that is
known to occur for natural invariants in positive characteristic~\cite{Hauser2008}, see section \ref{sec:background}.

\subsection{Intrinsic ranking functions and the absence of memory}
A natural idea when confronted with plateaux is to add stateful corrections:
for instance, to store the previous best rank and penalize returning to it, or to smooth oscillations
by a moving average of recent feature values.
However, our setting explicitly forbids such mechanisms: the ranking function must be pure
(no dependence on history, global variables, or randomness).
This constraint mirrors the mathematical requirement that a termination invariant should be an
intrinsic function of the current singularity data.

The way out is to enrich the \emph{intrinsic} state.
In genuine resolution algorithms the exceptional divisor carries the memory of past blow-ups;
accordingly, one of the most important shifts in our experiments was to make several feature
definitions \emph{boundary-aware} and to add new boundary features.
Concretely, the final successful rankers use:
\begin{itemize}[leftmargin=*]
\item a nontrivial boundary count and boundary multiplicity sum (features $f_4$ and $f_{25}$),
\item a boundary-aware $E$-order proxy (feature $f_8$),
\item a corrected boundary-aware weighted order proxy $w$-ord (feature $f_{14}$),
\end{itemize}
together with Jacobian and Hilbert-Samuel proxies (features $f_{21}$-$f_{24}$) that act as
plateau breakers in purely inseparable regimes.

Empirically, once these boundary and structural signals were present, AlphaEvolve was able to find
rankers whose trajectories almost never stagnate: on the final $100$-case benchmark the maximum plateau
length dropped to $2$ and all bounded-delay violations vanished (Section~\ref{subsec:discretize}).
This supports the general heuristic that in positive characteristic one should not expect order-type
invariants alone to control termination; boundary and Frobenius-sensitive data must be built into the
state from the start.

\subsection{Why the broad benchmark was too ambitious}
\label{sec:toobroad}
Our first benchmark collection mixed several characteristics ($p\in\{2,3,5,7\}$), several ambient
dimensions, and many qualitatively different families of hypersurface singularities. 
This was useful for quickly surfacing failure modes, but it turned out to be too ambitious for
feature search: the optimization landscape was dominated by mutually incompatible behaviors.

A representative best candidate produced by AlphaEvolve after a long run on this broad benchmark
still had very poor global behavior:
it solved only $14$ instances outright, accumulated a total of $1668$ bounded-delay violations,
and exhibited a maximum plateau length of $28$ steps.
Moreover, the failures were not specific to a single pathology; rather, different subfamilies
failed for different reasons.

This experience motivated a change in experimental strategy.
We restricted attention to a fixed ambient dimension $4$ and a fixed characteristic $p=3$,
and we built a \emph{focused} benchmark of $100$ instances of the form "monic purely inseparable
$z^3$ with base competition", designed to remain nontrivial while being sufficiently coherent for
evolutionary search.
Within this restricted regime, both the continuous and discretized rankers achieve perfect scores, and the remaining failure modes are subtle enough to serve as
targets for future refinement rather than as overwhelming noise.

\section{Focused benchmark for dimension 4 and characteristic 3}\label{sec:results}

In this section we present the results of our experiments. We find counterexample of initial perfect ranking function evolved on a first test set, and describe how this let to an improvement of the experiment. We formulate two conjectures which we consider the biggest results of this work.  
\subsection{First experiment: perfect ranking function on initial test set} \label{subsec:experiment1}
We constructed a robust test set of $71$ hypersurface singularites, including:
\begin{itemize}[leftmargin=2em]
\item Frobenius-flat plateau crosses with base exponents multiples of $3$ (maximizing Jacobian vanishing),
\item permutations of monomial ordering to prevent tie-breaking artifacts in center selection,
\item immediate and delayed shade surrogates: mixed monomials whose degrees enter the critical window $[p,2p)$ after several transforms,
\item initial-form perturbations of the form $z^2\cdot(\text{base})$ to break inseparability at order $3$,
\item toroidal/quasi-monomial cores with small oblique perturbations,
\item optional order $p^2$ holdouts (e.g.\ $z^9$) to probe a harder regime.
\end{itemize}

On this focused benchmark, after 26 hours AlphaEvolve found (see Appendix ~A) a pure two-component ranking function 
\[
\mathcal{R}(x) = (c_1(x), c_{\mathrm{combined}}(x))
\]
where the evolved program internally computes components $c_2,c_3,c_4,c_5$ and then forms the weighted sum
\[
c_{\mathrm{combined}}
=
284669250\,c_2
+
5581750\,c_3
+
250\,c_4
+
c_5
\]
Here (writing $f_i$ for the features):
\begin{align*}
c_1 &= 
\begin{cases}
0 & \text{if } f_9 = 1\ \text{(monomial phase)},\\
f_0 + 0.25 & \text{if } f_9 = 0,
\end{cases}
\\[2mm]
c_2 &= f_{14} \qquad\text{(boundary-aware } w\text{-ord proxy)},\\[1mm]
c_3 &= 50\cdot \tanh \left(\frac{1}{5}\Big(f_{21} + 0.1 f_{19} + 0.05 f_{20} + f_{10} - J + P\Big)\right),\\
&\hspace{1.2cm}\text{with }J=(1-f_{23})\left(1+\frac{f_{22}}{5}\right),\quad
P=-5\cdot \operatorname{atan2}\!\left(\frac{f_{24}}{10},\frac{f_{21}}{25}\right),\\[2mm]
c_4 &= 0.15 f_1 - 1.5 f_7 - \exp(0.1 f_{25}) + 0.2 f_8,\\[1mm]
c_5 &= 0.5 f_5 + 0.5 f_{18} - 0.1 f_4.
\end{align*}

On the focused $71$-case benchmark, the scalarized two-component ranking function $R$ achieves zero violations with window $m=5$:

\begin{center}
\begin{tabular}{@{}lr@{}}
\toprule
Metric & Value \\
\midrule
Total violations  & $0$ \\
Solved cases (full stage) & $71/71$ \\
Total local increases & $206.5$ \\
Max plateau length & $2$ \\
\bottomrule
\end{tabular}
\end{center}

We see the weights in $c_{combined} $ in practice suggest a lexicographic order, so it is natural to ask whether the explicit lex tuple 
\begin{equation}\label{rlex}
    \mathcal R_{\mathrm{lex}}=(c_1,c_2,c_3,c_4,c_5)
\end{equation}
is itself a bounded-delay Lyapunov function for broader classes of inputs.
The answer is \emph{no} without further restrictions: outside the focused test family, one can force the second component $c_2$ (the boundary-aware weighted order proxy) to attain its minimum value $0$ transiently, after which lexicographic improvement becomes impossible unless $c_2$ returns to $0$ and one improves $c_3$.
This can lead to arbitrarily long frozen best-so-far plateaux. Here is a counterexample we found using GPT 5.2 Pro for $m=10$.

\begin{lemma}[Counterexample to bounded-delay for the lex tuple]
\label{prop:clean-lex-counterexample}
Consider the following hypersurface with elimination variable $z$:
\[
f = z^3 + x^7 + z^5 y^2 + z^3 x^2 w^4 y^6 + z^3 x^2 w^6 y^3 + z^4 w^6 y^5.
\]
In the toy canonical simulator of Section~\ref{sec:setup}, the clean lex tuple rank $\mathcal R_{\mathrm{lex}}$ in \eqref{rlex} violates the bounded-delay condition not only for $m=5$ but also for $m=10$.
\end{lemma}

\begin{proof}[Proof sketch]
We verified this computationally by simulating $30$ canonical steps and tracking $\mathcal R_{\mathrm{lex}}$.
Along the resulting trajectory, the weighted-order component $c_2$ drops to $0$ at step $2$, while subsequent steps have $c_2>0$ for an extended stretch.
Since $c_2\ge 0$ by construction, the best-so-far lex value becomes locked at $c_2=0$ and no improvement occurs for more than $10$ consecutive steps, meaning  a violation for $m=10$.
\end{proof}

The counterexample above uses monomials with additional $z$-powers and is far outside the monic purely inseparable Weierstrass shape of our focused benchmark (pure $z^3$ plus base competition and mild $z^2$-perturbations).
It remains plausible that $\mathcal R_{\mathrm{lex}}$ is a delayed ranking function on a restricted family, for example polynomials of the form
\[
z^3 + g(x,y,w) + z^2 h(x,y,w),
\]
with $g$ in a plateau/cross regime.
We therefore record the following as a working conjecture.

\begin{conjecture}[Clean lex tuple on the focused Weierstrass family]
\label{conj:clean-lex-weierstrass}
On the monic Weierstrass family $z^3+g+z^2h$ described above, the clean lex tuple $\mathcal R_{\mathrm{lex}}=(c_1,c_2,c_3,c_4,c_5)$ satisfies the bounded-delay descent condition for some universal window $m$ (empirically $m=5$).
\end{conjecture}

We finally add some intuition on the evolved ranking function. Component $c_3$ is seems to be the most novel: it blends a Hilbert-Samuel signal ($f_{21}$), base max-locus complexity ($f_{19},f_{20}$),
an initial inseparability flag ($f_{10}$), Jacobian information ($f_{22},f_{23}$), and a bounded "phase potential" built from
$(f_{21},f_{24})$ via $\operatorname{atan2}$ and then saturated by $\tanh$.
This appears to provide a stable plateau breaker when $f_0$ and $w$-ord plateau.

\subsection{Improved experiment:  discretization}\label{subsec:discretize}

The evolved ranker of Section  returns a pair $(c_1,c_{\mathrm{comb}})\in \mathbb R^2$ whose
second coordinate is a weighted sum of four logically distinct components.
For mathematical use it is preferable to work with an explicit \emph{lexicographic tuple} whose
entries correspond to recognizable geometric signals (compare the characteristic-zero invariants
in Section~2.1).
This motivates two additional steps: we have rewritten the evolve block so that it evolves a  tuple $(c_1,c_2,c_3,c_4,c_5)$ and a discretization  map $\Pi: \mathbb{R}^5 \to \mathbb N^5$ so that the target is a well-founded ordered set.

After 24 hours of running, in the final successful lexicographic version (Appendix~B) the components have the following evolved form:


\[
\begin{aligned}
c_1 &=
\begin{cases}
0, & \text{if } f_9=1,\\
f_0, & \text{if } f_9\neq 1,
\end{cases}
\\[4pt]
c_2 &= \tfrac12 f_{14}+\tfrac12 f_{21}+0.05\,f_1+0.01\,f_5,\\[4pt]
c_3 &= f_{10}+f_{19}+0.1\,f_{20},\\[4pt]
c_4 &=-\Bigl(4\,f_{24}^3+f_{25}+5(1-f_{23})f_{24}+10\,f_{10}\,f_{24}(1-f_{23})\Bigr),\\[4pt]
c_5 &= f_{18}+\tfrac12 f_8.
\end{aligned}
\]
The evolved discretization function 
$\Pi\colon \mathbb R^5\to\mathbb N^5$ is the following:
\begin{align*}
d_1 &:= \lfloor c_1\rfloor,\\
d_2 &:= \lfloor 100\cdot c_2\rfloor,\\
d_3 &:= \lfloor 10\cdot (c_3+50)\rfloor,\\
d_4 &:= 5000-\Big\lfloor 100\cdot \log\bigl(1+\max(0,-c_4)\bigr)\Big\rfloor,\\
d_5 &:= \lfloor 10\cdot (c_5+20)\rfloor.
\end{align*}
The logarithmic compression in $d_4$ is meant to make very negative $c_4$ values comparable on a
finite scale, while preserving the monotonicity, i.e. more negative $c_4$ is better.

On the focused $71$-case benchmark (dim $4$, $p=3$, monic purely inseparable $z^3$ with base
competition) the discretized rank
\[
\mathcal R_{\mathrm{disc}} := \Pi\circ (c_1,c_2,c_3,c_4,c_5)
\]
achieves a perfect score under the bounded-delay criterion with $m=5$:
\begin{center}
\begin{tabular}{@{}lr@{}}
\toprule
Metric & Value \\
\midrule
Total violations  & $0$ \\
Solved cases (full stage) & $71/71$ \\
Total local increases & $344.5$ \\
Max plateau length & $2$ \\
\bottomrule
\end{tabular}
\end{center}

The total number of non-improving steps increased to $344.5$. Nevertheless, the discretized lex tuple retains the key qualitative behavior: frequent improvements and no extended stagnation. Some intuitive explanation for the roles of the components are as follows:
\begin{itemize}[leftmargin=*]
\item $c_1$ enforces the \emph{monomial phase rule}: $c_1=0$ in monomial phase and $c_1>0$ otherwise.
\item $c_2$ is the primary "plateau breaker", balancing the boundary-aware weighted order proxy
($w$-ord) with the Hilbert-Samuel base proxy and small contributions from elimination order and shade.
\item $c_3$ tracks inseparability and the dimension/complexity of the maximal locus in the base.
\item $c_4$ is an inverted "depth/complexity" term, dominated by a cubic function of the $p$-adic
depth and by the boundary multiplicity sum, and further amplified when the Jacobian vanishes.
\item $c_5$ is a low-priority tie-breaker penalizing wildness and (boundary-aware) $E$-order.
\end{itemize}

The following explicit counterexample, found by GPT 5.2 Pro shows that the present
$\Pi$ does \emph{not} extend to arbitrary inputs of the same ambient parameters.

\begin{lemma}[A counterexample outside the benchmark]\label{thm:counterexample}
Let $p=3$ and consider the hypersurface with elimination variable $z$:
\[
z^3 + x^{12} + y^{6} + w^{9}y^{4} + x^{9}y^{8}w^{10}.
\]
Then, under the toy canonical blow-up simulator, the discretized rank
$\mathcal R_{\mathrm{disc}}$ above violates the bounded-delay condition with window $m=5$.
\end{lemma}

\begin{proof}[Proof sketch]
A direct simulation shows that the discretized ranks at steps begin as
\begin{multline}
(3,4280,531,5000,220),\ (3,4280,531,5000,220),\ (3,4130,522,5000,220),\\ \nonumber (3,965,531,5000,220),\ (3,915,522,5000,220),
\end{multline}
followed by five consecutive steps whose ranks are \emph{not} lexicographically smaller than the
best-so-far rank $(3,915,522,5000,220)$, despite improvements in later components.
At step $9$ one has $(3,999,511,4770,210)$: here $d_3$ and $d_4$ improve, but $d_2$ increases and blocks
lexicographic improvement.
Thus no improvement occurs in a window of length $5$, giving a bounded-delay violation.
\end{proof}

Lemma \ref{thm:counterexample} indicates that the discretized lex order is too sensitive to
the second component on unrestricted input families.
A natural modification is to either (i) restrict the input class to the structured families used in
our benchmark (bounded-degree mixed perturbations and Frobenius-flat base competition), or (ii) coarsen
the discretization of $c_2$ so that spurious oscillations of $d_2$ do not mask genuine progress in
$(d_3,d_4,d_5)$.

\subsection{Third experiment: a counterexample-driven improvement}
\label{subsec:disc-100}

In the final series of experiments we added the counterexample of Lemma \ref{thm:counterexample} to the benchmark, and we further expanded the
test set so as to stress (i) long-delay heavy-tail mixed terms, (ii) sensitivity to ordering and tie-breaking
in the center-selection heuristic, (iii) delayed shade spikes (a toy analogue of kangaroo-type behaviour), and
(iv) a small holdout regime of order $p^2$ examples.
The resulting dim$=4$, $p=3$ test set contains $100$ test ideals (Appendix~A.4).

We also replaced the linear primary objective by a \emph{saturated} score on the full stage:
\[
\mathrm{Score} = 2N_{\mathrm{solved}} - \sum_{i=1}^{|\mathcal B_{100}|} \tanh \Bigl(\frac{\mathrm{violations}_i}{10}\Bigr).
\]
This ensures that solving an additional case (awareded by $+2$) always dominates any amount of penalty reduction on a
single failing case (worth at most $1$), and empirically stabilizes the search when new adversarial cases are added.

We initialized AlphaEvolve with the discretized lex rank from Appendix~B.2 (the solution in the previous experiment). This means that our evolve block started evolving from the learned discretized lex tuple from the previous experiment. After approximately $4$ hours, AlphaEvolve discovered a new five-component raw rank function
\[\mathcal{R}_{100}=\Pi \circ (c_1,c_2,c_3,c_4,c_5)\] 
(see Appendix~B.3) which, under the \emph{same} discretization map $\Pi$ from
Section \ref{subsec:discretize} and the same bounded-delay window $m=5$, achieves perfect performance 
\begin{center}
\begin{tabular}{@{}lr@{}}
\toprule
Metric & Value \\
\midrule
Total violations  & $0$ \\
Solved cases (full stage) & $100/100$ \\
Total local increases & $285$ \\
Max plateau length & $2$ \\
\bottomrule
\end{tabular}
\end{center}
In particular, the counterexample of Lemma \ref{thm:counterexample} does no longer violates this ranking function. Compared to the $\mathcal R_{\mathrm{disc}}$
discretized solution, the new ranker makes two qualitative changes.
The components of the new naking function (before discretization) are:
\[
\begin{aligned}
c_1 &=
\begin{cases}
0, & \text{if } f_9=1,\\
f_0, & \text{if } f_9\neq 1,
\end{cases}
\\[6pt]
c_2 &= 1.0\,f_{14}+0.1\,f_{21}+0.1\,f_{1}+0.8\,f_{23}+0.5\,f_{7}+0.2\,f_{17},
\\[6pt]
c_3 &= f_{10}+2.0\,f_{19}+0.5\,f_{20}+0.1\,f_{4}+0.2\,f_{12}+0.1\,f_{13}.
\end{aligned}
\]

For the fourth component, define the intermediate quantities (as they stand in the evolved block):
\[
\begin{aligned}
A &:= 10.0\,f_{24}^{2}+5.0\,f_{25}
\qquad\text{(baseline\_acc\_term)},\\[4pt]
J &:= f_{6}+(1.0-f_{23})+f_{12}+f_{13}
\qquad\text{(jac\_problem\_signal)},\\[4pt]
\alpha &:= \max\!\left\{0,\ \tanh\!\left(J+\frac{f_{21}}{1.0+f_{22}}\right)\right\}
\qquad\text{(jac\_problem\_activation)},\\[6pt]
W &:= f_{10}+f_{18}+f_{5}+f_{19}f_{20}+f_{4}
\qquad\text{(insep\_wild\_signal)},\\[4pt]
\beta &:= \max\!\left\{0,\ \tanh\!\left(\frac{W}{5.0}\right)\right\}
\qquad\text{(insep\_wild\_activation)},\\[6pt]
\gamma &:= \max\!\left\{0,\ \tanh\!\left(\frac{f_{24}+f_{25}+f_{4}}{10.0}\right)\right\}
\qquad\text{(effort\_pressure)},\\[6pt]
\kappa &:= 1000.0\left(1.0+\tanh\!\left(\frac{A+f_{4}+f_{5}+f_{18}}{100.0}\right)\right)
\qquad\text{(catastrophe\_magnitude\_coeff)},\\[6pt]
\sigma &:= 0.01+0.5\,\alpha+0.5\,\beta+0.1\,\gamma
\qquad\text{(combined\_pathology\_activation\_strength)},\\[6pt]
P &:= \kappa\,\exp(\sigma)
\qquad\text{(contextual\_pathology\_interaction\_term)}.
\end{aligned}
\]
Then
\[
c_4 = -\bigl(A+P\bigr).
\]
Finally,
\[
c_5 = f_{18}+f_{5}+0.5\,f_{8}+2.0\,f_{6}+0.1\,f_{15}-0.1\,f_{22}.
\]
Note that:
\begin{itemize}[leftmargin=*]
\item The component $c_2$ is again dominated by the boundary-residual weighted-order proxy $f_{14}$,
but it re-introduces a substantial Jacobian-density term (the nonzero-partials proxy $f_{23}$) and a Newton-slope term
$f_7$. Empirically this reduces oscillation in the discretized component $\Pi_1$ on heavy-tail instances.

\item The accumulator component $c_4$ contains a \emph{catastrophe term}:
a large negative contribution built from an exponential of a blended activation of Jacobian pathologies,
inseparability/wildness, and accumulated boundary/Frobenius-depth.
Under the discretization $\Pi_3 = 5000 - \lfloor 100\log(1-\!c_4)\rfloor$, this produces reliable improvements in
$\Pi_3$ exactly in regimes where the earlier ranker could stall for more than $m=5$ steps.
\end{itemize}

Lemma \ref{thm:counterexample} shows that the heavy-tail mixed terms can force delay gaps
that are not detected by a purely "low-order" discretized invariant.
The $100$-case run suggests that incorporating additional Jacobian-density information and an explicit
(non-linear) interaction with Frobenius-depth and boundary accumulation can repair this behaviour.

\begin{conjecture}[Discretized ranking function]
\label{conj:disc-robust}
Let $\Pi$ be the discretization map of Section~\ref{subsec:discretize}. 
Then there exists a five-component raw rank function $R=(c_1,c_2,c_3,c_4,c_5)$ built from the intrinsic invariants $f_0,\ldots, f_{25}$ such that $\Pi \circ R$ is a bounded-delay ranking function for $\mathrm{dim}=4$, $p=3$ monic purely inseparable hypersurface singularities. 
Moreover, the explicit ranker $\mathcal{R}_{100}$ from Section \ref{subsec:disc-100} is conjectured to generalize to families obtained by adding further mixed terms of arbitrarily high degree.
\end{conjecture}

We emphasize that Conjecture~\ref{conj:disc-robust} is still only supported by finite computational evidence,
but it points toward a plausible form of a delayed lexicographic invariant in positive characteristic.

\appendix

\section{Test singularities}

\subsection{The broad benchmark}
In our initial experiments to find ranking functions, we used a broad benchmark consisting of singularities across various dimensions and characteristics It contained two types of inputs:
\begin{enumerate}
  \item Hand-designed prototypes (24 singularities), intended to isolate specific characteristic-$p$ failure modes: purely inseparable plateaux, immediate order drops, oblique/mixed perturbations, etc.
  \item Randomized surrogates (60 cases), generated by sampling exponent patterns in dimensions
  $4,5,6$ to increase diversity.
\end{enumerate}

The goal of this benchmark is adversarial coverage rather than trying to faithfully simulate typical or naturally occurring singularities. Although the simulator is only a toy model, it forces the ranking function to handle several difficult situations: (i) Jacobians that become flat in characteristic $p$,
(ii) multiple base monomials with similar orders that compete with each other, and
(iii) the unexpected appearance of low-degree mixed terms after repeated chart reductions.
Table~\ref{tab:legacy-structured} lists the 24 hand-designed cases.

\newpage

\begin{longtable}{@{}p{0.22\textwidth} r r p{0.28\textwidth} p{0.18\textwidth}@{}}
\caption{Multi-dimensional benchmark of prototype cases. Names and notes generated by Gemini}\label{tab:legacy-structured}\\
\toprule
Name & $p$ & dim & Monomials & Notes \\
\midrule
\endfirsthead
\toprule
Name & $p$ & dim & Monomials  & Notes \\
\midrule
\endhead
plateau\_line\_A3 & 5 & 3 & $z^{5} + x^{10}$ & Equimultiple line \{x=z=0\}; first blow-up plateaus (order stays p). \\
drop\_line\_A3 & 5 & 3 & $z^{5} + x^{7}$ & First blow-up in x-chart drops order from p to (p+2)-p=2. \\
wild\_surrogate\_A3 & 5 & 3 & $z^{5} + x^{5}y + x^{10}$ & Toy wild case: keeps order p after first step; adds a shade-like penalty. \\
drop\_plane\_A4 & 5 & 4 & $z^{5} + x^{5}w^{4}$ & Singular plane \{x=z=0\}; first blow-up in x-chart drops order to 1. \\
plateau\_cross\_A3 & 5 & 3 & $z^{5} + x^{10} + y^{10} + x^{5}y^{5}$ & Max locus is the y-axis $\cup$ x-axis; canonical center selects the component meeting elim conditions. \\
oblique\_surrogate\_A3 & 5 & 3 & $z^{5} + x^{5}y^{4} + y^{10}$ & Surrogate for oblique/kangaroo stress: may force non-decrease in early steps. \\
non\_monic\_z\_A3 & 5 & 3 & $z^{4}x + y^{10}$ & Non-monic in z; elimination proxy stresses f2 \\
binomial\_pair\_A4 & 5 & 4 & $z^{5} + x^{10} + y^{5} + w^{5}$ & Toy binomial ideal as single combined polynomial proxy \\
AS\_flavor\_A3 & 5 & 3 & $z^{5} + z + x^{5}y^{4}$ & Artin-Schreier surrogate; mixed separable/inseparable \\
jac\_vanish\_A3 & 5 & 3 & $z^{10} + x^{15} + y^{10}$ & All exponents multiples of p; Jacobian-vanish flag=1 \\
monomial\_control\_A3 & 5 & 3 & $x^{7}y^{5}$ & Monomial-phase control (no z) \\
param\_drop\_a6\_A3 & 5 & 3 & $z^{5} + x^{6}$ & Parameterized immediate-drop family \\
param\_drop\_a7\_A3 & 5 & 3 & $z^{5} + x^{7}$ & Parameterized immediate-drop family \\
param\_drop\_a9\_A3 & 5 & 3 & $z^{5} + x^{9}$ & Parameterized immediate-drop family \\
A4\_drop\_mixed & 5 & 4 & $z^{5} + x^{5}w^{4}$ & Classic 4D mixed; expected order drop \\
A4\_plateau\_plus\_mix & 5 & 4 & $z^{5} + x^{10} + x^{5}y + w^{5}$ & Plateau tendency with extra base monomial \\
A4\_nonmonic\_z & 5 & 4 & $z^{4}x^{2} + y^{10} + w^{5}$ & Non-monic in z; elimination stress \\
A4\_AS\_flavor & 5 & 4 & $z^{5} + z + x^{5}y^{4} + w^{9}$ & Artin-Schreier-like with extra high base power \\
A4\_binomial\_toroidal & 5 & 4 & $z^{5} + x^{5}y^{5}$ & Toroidal/binomial-type; monomial after elimination \\
A5\_drop\_two\_params & 5 & 5 & $z^{5} + x^{5}u^{4} + y^{10}$ & 5D mixed with two base parameters \\
A5\_wild\_oblique & 5 & 5 & $z^{5} + x^{5}y^{4}u + v^{10}$ & Oblique with an extra factor u \\
A5\_cross\_competition & 5 & 5 & $z^{5} + x^{10} + y^{10} + x^{5}y^{5} + u^{5}v^{4}$ & Competing centers and mixed base term \\
A6\_multi\_mixed & 5 & 6 & $z^{5} + x^{5}u^{4} + y^{5}v^{4} + w^{9}$ & Two mixed pairs plus high-power base term \\
A6\_nonmonic\_wild & 5 & 6 & $z^{4}x^{2} + x^{5}y^{4}v + u^{10} + w^{5}$ & Non-monic + oblique + extra base powers \\
\bottomrule
\end{longtable}

\subsection{The focused $\dim=4$, $p=3$ benchmark (71 cases)}
The second phase of experimentation restricts to $\mathbb{A}^4$ in characteristic $3$ and to
\emph{monic purely inseparable} presentations with leading term $z^3$ plus competing base terms.
This restriction is motivated by two considerations:
(i) it isolates the characteristic-$p$ plateau regime that is most challenging for
Lyapunov-style ranking functions, and (ii) it allows us to densely populate the benchmark
with targeted adversarial variants (tie permutations, delayed shade spikes, $z^2$-perturbations)
while keeping the overall evaluation cost manageable. The 71 test singularities are the disjoint union of:
\begin{itemize}
  \item a focused base-competition set of 20 cases (Appendix~\ref{tab:focused20}),
  built from plateau lines, cross competitions, and mixed degree-$6$ perturbations; and
  \item an adversarial tricky set of 51 cases (Appendix~\ref{tab:tricky51}),
  specifically constructed to challenge tie-breaking, shade spikes, initial-form inseparability, and order-$p^2$ regimes.
\end{itemize}

\begin{longtable}{@{}p{0.30\textwidth} p{0.32\textwidth} p{0.30\textwidth}@{}}
\caption{Focused $\dim=4$, $p=3$ test set: base-competition cases (20). Names and notes generated by Gemini}\label{tab:focused20}\\
\toprule
Name & Monomials & Notes \\
\midrule
\endfirsthead
\toprule
Name & Monomials & Notes \\
\midrule
\endhead
p3\_A4\_plateau\_line\_x6 & $z^{3} + x^{6} + w^{6}$ & Monic z\textasciicircum{}3 with a single large base power (plateau-prone) plus w\textasciicircum{}6. \\
p3\_A4\_plateau\_line\_x9 & $z^{3} + x^{9} + w^{6}$ & Monic z\textasciicircum{}3 with a single large base power (plateau-prone) plus w\textasciicircum{}6. \\
p3\_A4\_plateau\_line\_x12 & $z^{3} + x^{12} + w^{6}$ & Monic z\textasciicircum{}3 with a single large base power (plateau-prone) plus w\textasciicircum{}6. \\
p3\_A4\_cross\_x6\_y6 & $z^{3} + x^{6} + y^{6} + w^{6}$ & Cross/competition among base pure powers at multiples of p. \\
p3\_A4\_cross\_x6\_y9 & $z^{3} + x^{6} + y^{9} + w^{6}$ & Cross/competition among base pure powers at multiples of p. \\
p3\_A4\_cross\_x6\_y12 & $z^{3} + x^{6} + y^{12} + w^{6}$ & Cross/competition among base pure powers at multiples of p. \\
p3\_A4\_cross\_x9\_y6 & $z^{3} + x^{9} + y^{6} + w^{6}$ & Cross/competition among base pure powers at multiples of p. \\
p3\_A4\_cross\_x9\_y9 & $z^{3} + x^{9} + y^{9} + w^{6}$ & Cross/competition among base pure powers at multiples of p. \\
p3\_A4\_cross\_x9\_y12 & $z^{3} + x^{9} + y^{12} + w^{6}$ & Cross/competition among base pure powers at multiples of p. \\
p3\_A4\_cross\_x12\_y6 & $z^{3} + x^{12} + y^{6} + w^{6}$ & Cross/competition among base pure powers at multiples of p. \\
p3\_A4\_cross\_x12\_y9 & $z^{3} + x^{12} + y^{9} + w^{6}$ & Cross/competition among base pure powers at multiples of p. \\
p3\_A4\_cross\_x12\_y12 & $z^{3} + x^{12} + y^{12} + w^{6}$ & Cross/competition among base pure powers at multiples of p. \\
p3\_A4\_cross\_square\_6 & $z^{3} + x^{6} + y^{6} + x^{3}y^{3} + w^{6}$ & Plateau/cross family with mixed p-toroidal term x\textasciicircum{}3 y\textasciicircum{}3. \\
p3\_A4\_cross\_square\_6\_w9 & $z^{3} + x^{6} + y^{6} + x^{3}y^{3} + w^{9}$ & Same as cross\_square\_6 but with w\textasciicircum{}9 competitor. \\
p3\_A4\_plateau\_mix\_6a & $z^{3} + x^{6} + x^{4}y^{2} + w^{6}$ & Plateau base with mixed degree-6 perturbation x\textasciicircum{}4 y\textasciicircum{}2 (Jacobian not identically 0). \\
p3\_A4\_plateau\_mix\_6b & $z^{3} + y^{6} + x^{2}y^{4} + w^{6}$ & Plateau base with mixed degree-6 perturbation x\textasciicircum{}2 y\textasciicircum{}4. \\
p3\_A4\_cross\_mix\_6c & $z^{3} + x^{6} + y^{6} + x^{5}y + w^{6}$ & Cross with mixed perturbation x\textasciicircum{}5 y (wildness/jacobian density changes). \\
p3\_A4\_cross\_mix\_6d & $z^{3} + x^{9} + y^{6} + x^{7}y^{2} + w^{6}$ & Cross at (9,6) with mixed perturbation x\textasciicircum{}7 y\textasciicircum{}2. \\
p3\_A4\_cross\_mix\_6e & $z^{3} + x^{12} + y^{9} + x^{8}yw + w^{6}$ & Heavier cross (12,9) with small mixed x\textasciicircum{}8 y w. \\
p3\_A4\_pure\_cross\_666 & $z^{3} + x^{6} + y^{6} + w^{6}$ & Pure triple competition among base pure powers (classic plateau/cross). \\
\bottomrule
\end{longtable}

\begin{longtable}{@{}p{0.30\textwidth} p{0.32\textwidth} p{0.30\textwidth}@{}}
\caption{Focused $\dim=4$, $p=3$ test set: adversarial tricky cases (51).Notes and names generated by Gemini}\label{tab:tricky51}\\
\toprule
Name & Monomials  & Notes \\
\midrule
\endfirsthead
\toprule
Name & Monomials  & Notes \\
\midrule
\endhead
p3\_A4\_monomial\_control\_1 & $x^{7}y^{5}w^{4}$ & Monomial phase sanity check (no z). \\
p3\_A4\_monomial\_control\_2 & $x^{9}y^{6}$ & Monomial phase sanity check (no z, Frobenius-flat exponents). \\
p3\_A4\_cross\_Frob\_6\_6\_6 & $z^{3} + x^{6} + y^{6} + w^{6}$ & Purely inseparable (monic z\textasciicircum{}3) with Frobenius-flat competing base powers. \\
p3\_A4\_cross\_Frob\_9\_6\_6 & $z^{3} + x^{9} + y^{6} + w^{6}$ & Purely inseparable (monic z\textasciicircum{}3) with Frobenius-flat competing base powers. \\
p3\_A4\_cross\_Frob\_9\_9\_6 & $z^{3} + x^{9} + y^{9} + w^{6}$ & Purely inseparable (monic z\textasciicircum{}3) with Frobenius-flat competing base powers. \\
p3\_A4\_cross\_Frob\_9\_9\_9 & $z^{3} + x^{9} + y^{9} + w^{9}$ & Purely inseparable (monic z\textasciicircum{}3) with Frobenius-flat competing base powers. \\
p3\_A4\_cross\_Frob\_12\_9\_6 & $z^{3} + x^{12} + y^{9} + w^{6}$ & Purely inseparable (monic z\textasciicircum{}3) with Frobenius-flat competing base powers. \\
p3\_A4\_cross\_Frob\_12\_12\_6 & $z^{3} + x^{12} + y^{12} + w^{6}$ & Purely inseparable (monic z\textasciicircum{}3) with Frobenius-flat competing base powers. \\
p3\_A4\_cross\_Frob\_12\_12\_12 & $z^{3} + x^{12} + y^{12} + w^{12}$ & Purely inseparable (monic z\textasciicircum{}3) with Frobenius-flat competing base powers. \\
p3\_A4\_cross\_Frob\_15\_12\_9 & $z^{3} + x^{15} + y^{12} + w^{9}$ & Purely inseparable (monic z\textasciicircum{}3) with Frobenius-flat competing base powers. \\
p3\_A4\_tieperm\_6\_1 & $z^{3} + x^{6} + y^{6} + w^{6}$ & Same polynomial, different monomial ordering to stress tie-breaking invariance. \\
p3\_A4\_tieperm\_6\_2 & $z^{3} + x^{6} + w^{6} + y^{6}$ & Same polynomial, different monomial ordering to stress tie-breaking invariance. \\
p3\_A4\_tieperm\_6\_3 & $z^{3} + y^{6} + x^{6} + w^{6}$ & Same polynomial, different monomial ordering to stress tie-breaking invariance. \\
p3\_A4\_tieperm\_6\_4 & $z^{3} + y^{6} + w^{6} + x^{6}$ & Same polynomial, different monomial ordering to stress tie-breaking invariance. \\
p3\_A4\_tieperm\_6\_5 & $z^{3} + w^{6} + x^{6} + y^{6}$ & Same polynomial, different monomial ordering to stress tie-breaking invariance. \\
p3\_A4\_tieperm\_6\_6 & $z^{3} + w^{6} + y^{6} + x^{6}$ & Same polynomial, different monomial ordering to stress tie-breaking invariance. \\
p3\_A4\_tieperm\_9\_1 & $z^{3} + x^{9} + y^{9} + w^{9}$ & Same polynomial, different monomial ordering to stress tie-breaking invariance. \\
p3\_A4\_tieperm\_9\_2 & $z^{3} + x^{9} + w^{9} + y^{9}$ & Same polynomial, different monomial ordering to stress tie-breaking invariance. \\
p3\_A4\_tieperm\_9\_3 & $z^{3} + y^{9} + x^{9} + w^{9}$ & Same polynomial, different monomial ordering to stress tie-breaking invariance. \\
p3\_A4\_tieperm\_9\_4 & $z^{3} + y^{9} + w^{9} + x^{9}$ & Same polynomial, different monomial ordering to stress tie-breaking invariance. \\
p3\_A4\_tieperm\_9\_5 & $z^{3} + w^{9} + x^{9} + y^{9}$ & Same polynomial, different monomial ordering to stress tie-breaking invariance. \\
p3\_A4\_tieperm\_9\_6 & $z^{3} + w^{9} + y^{9} + x^{9}$ & Same polynomial, different monomial ordering to stress tie-breaking invariance. \\
p3\_A4\_tieperm\_12\_1 & $z^{3} + x^{12} + y^{12} + w^{12}$ & Same polynomial, different monomial ordering to stress tie-breaking invariance. \\
p3\_A4\_tieperm\_12\_2 & $z^{3} + x^{12} + w^{12} + y^{12}$ & Same polynomial, different monomial ordering to stress tie-breaking invariance. \\
p3\_A4\_tieperm\_12\_3 & $z^{3} + y^{12} + x^{12} + w^{12}$ & Same polynomial, different monomial ordering to stress tie-breaking invariance. \\
p3\_A4\_tieperm\_12\_4 & $z^{3} + y^{12} + w^{12} + x^{12}$ & Same polynomial, different monomial ordering to stress tie-breaking invariance. \\
p3\_A4\_tieperm\_12\_5 & $z^{3} + w^{12} + x^{12} + y^{12}$ & Same polynomial, different monomial ordering to stress tie-breaking invariance. \\
p3\_A4\_tieperm\_12\_6 & $z^{3} + w^{12} + y^{12} + x^{12}$ & Same polynomial, different monomial ordering to stress tie-breaking invariance. \\
p3\_A4\_immediate\_shade\_1 & $z^{3} + x^{9} + y^{6} + w^{6} + x^{2}y$ & Immediate shade/wildness stress: mixed monomial in degree window [3,5]. \\
p3\_A4\_immediate\_shade\_2 & $z^{3} + x^{9} + y^{6} + w^{6} + xy^{2}$ & Immediate shade/wildness stress: mixed monomial in degree window [3,5]. \\
p3\_A4\_immediate\_shade\_3 & $z^{3} + x^{9} + y^{6} + w^{6} + xyw$ & Immediate shade/wildness stress: mixed monomial in degree window [3,5]. \\
p3\_A4\_immediate\_shade\_4 & $z^{3} + x^{9} + y^{6} + w^{6} + x^{2}yw$ & Immediate shade/wildness stress: mixed monomial in degree window [3,5]. \\
p3\_A4\_immediate\_shade\_5 & $z^{3} + x^{9} + y^{6} + w^{6} + x^{2}y^{2}w$ & Immediate shade/wildness stress: mixed monomial in degree window [3,5]. \\
p3\_A4\_immediate\_shade\_6 & $z^{3} + x^{9} + y^{6} + w^{6} + x^{3}y^{2}$ & Immediate shade/wildness stress: mixed monomial in degree window [3,5]. \\
p3\_A4\_kangaroo\_delay\_step\_1\_deg4 & $z^{3} + x^{12} + y^{6} + w^{6} + x^{5}yw$ & Delayed-shade spike surrogate: mixed term enters shade window after several x-chart reductions. \\
p3\_A4\_kangaroo\_delay\_step\_2\_deg4 & $z^{3} + x^{15} + y^{6} + w^{6} + x^{8}yw$ & Delayed-shade spike surrogate: mixed term enters shade window after several x-chart reductions. \\
p3\_A4\_kangaroo\_delay\_step\_3\_deg4 & $z^{3} + x^{18} + y^{6} + w^{6} + x^{11}yw$ & Delayed-shade spike surrogate: mixed term enters shade window after several x-chart reductions. \\
p3\_A4\_kangaroo\_delay\_step\_4\_deg4 & $z^{3} + x^{21} + y^{6} + w^{6} + x^{14}yw$ & Delayed-shade spike surrogate: mixed term enters shade window after several x-chart reductions. \\
p3\_A4\_kangaroo\_delay\_step\_5\_deg4 & $z^{3} + x^{24} + y^{6} + w^{6} + x^{17}yw$ & Delayed-shade spike surrogate: mixed term enters shade window after several x-chart reductions. \\
p3\_A4\_kangaroo\_double\_delay\_1\_3 & $z^{3} + x^{18} + y^{6} + w^{6} + x^{5}yw + x^{11}y^{2}w$ & Two delayed-shade mixed terms,repeated spike opportunities (kangaroo-like). \\
p3\_A4\_kangaroo\_double\_delay\_2\_4 & $z^{3} + x^{21} + y^{6} + w^{6} + x^{8}yw + x^{14}y^{2}w$ & Two delayed-shade mixed terms, repeated spike opportunities (kangaroo-like). \\
p3\_A4\_kangaroo\_double\_delay\_3\_5 & $z^{3} + x^{24} + y^{6} + w^{6} + x^{11}yw + x^{17}y^{2}w$ & Two delayed-shade mixed terms, repeated spike opportunities (kangaroo-like). \\
p3\_A4\_z2\_initial\_perturb\_1 & $z^{3} + z^{2}x + x^{9} + y^{6} + w^{6}$ & Initial-form perturbation: z\textasciicircum{}2*(base) term breaks pure inseparability of initial degree-3 part. \\
p3\_A4\_z2\_initial\_perturb\_2 & $z^{3} + z^{2}y + x^{9} + y^{6} + w^{6}$ & Initial-form perturbation: z\textasciicircum{}2*(base) term breaks pure inseparability of initial degree-3 part. \\
p3\_A4\_z2\_initial\_perturb\_3 & $z^{3} + z^{2}xy + x^{9} + y^{6} + w^{6}$ & Initial-form perturbation: z\textasciicircum{}2*(base) term breaks pure inseparability of initial degree-3 part. \\
p3\_A4\_z2\_initial\_perturb\_4 & $z^{3} + z^{2}xw + x^{9} + y^{6} + w^{6}$ & Initial-form perturbation: z\textasciicircum{}2*(base) term breaks pure inseparability of initial degree-3 part. \\
p3\_A4\_z2\_initial\_perturb\_5 & $z^{3} + z^{2}x^{2}y + x^{9} + y^{6} + w^{6}$ & Initial-form perturbation: z\textasciicircum{}2*(base) term breaks pure inseparability of initial degree-3 part. \\
p3\_A4\_toroidal\_plus\_oblique\_1 & $z^{3} + x^{3}y^{3}w^{3} + x^{2}y + w^{9}$ & Toroidal base term + small oblique/mixed perturbation (breaks clean monomial behavior). \\
p3\_A4\_toroidal\_plus\_oblique\_2 & $z^{3} + x^{6}y^{3} + x^{3}y^{2} + w^{12}$ & Quasi-monomial base + oblique perturbation (designed to create non-monotone secondary features). \\
p3\_A4\_order9\_Frob\_flat & $z^{9} + x^{18} + y^{18} + w^{18}$ & Order $p^2$ case: $z^9$ plus Frobenius-flat base powers (holdout-style harder regime). \\
p3\_A4\_order9\_with\_shade & $z^{9} + x^{27} + y^{18} + w^{18} + x^{14}y^{2}w$ & Order p\textasciicircum{}2 with a mixed term that may drift toward the shade window under repeated x-chart reductions. \\
\bottomrule
\end{longtable}

\paragraph{Complete JSON manifest (71 cases).}
The full machine-readable manifest of the focused benchmark is included below.

\subsection{Extending the dim $4$, $p=3$ benchmark to $100$ cases}
\label{app:benchmark100}

For the final counterexample-driven experiment of Section~\ref{subsec:disc-100} we enlarged the
tricky dim $4$, $p=3$ test set decribed in Appendix~A.3 by $29$ additional instances, keeping the focused
$20$-case family unchanged. The additional cases were chosen to target failure types given by Theorem~\ref{thm:counterexample} and by
manual inspection of near-violations in earlier runs.  They fall into a small number of parametric families:

\begin{enumerate}[leftmargin=*]
\item Heavy-tail mixed terms (6 cases).
These include the explicit counterexample of Theorem~\ref{thm:counterexample} together with variations obtained by permuting exponents among variables and by lowering the mixed-term degree
to the boundary regime (a single degree-$15$ mixed monomial), plus two shade-boundary examples
with mixed terms of degrees $8$ and $9$, and a deep variable case where $w$ appears only in very high degree.
(See the \texttt{p3\_A4\_counter\_example\_*}, \texttt{p3\_A4\_shade\_boundary\_*}, and \texttt{p3\_A4\_deep\_variable\_w}
instances in the evaluation harness.)
\item Fermat-like singularities (5 cases).
For $i=1,\dots,5$ we include
\[
z^3 + x^6 + y^6 + x^3y^3w^3 + w^{3i+1},
\]
which combines a purely inseparable Fermat core with a small oblique perturbation.
\item Weighted-homogeneous plateaux (2 cases).
We include a weighted plateau
\[
z^3 + x^6 + y^9 + w^{18},
\]
together with a perturbation by an additional mixed term $x\,w^{15}$ of the same weighted degree.

\item We include five cases of the form
\[
z^3 + x^{a} + y^{b} + w^{12},
\]
with exponent pairs $(a,b)\in\{(7,5),(11,4),(13,2),(8,5),(10,3)\}$.

\item Extreme imbalance (1 case). A single singularity
\[
z^3 + x^4 + y^{100} + w^6
\]
to test behavior under widely separated base degrees.

\item Dense mixed perturbations (10 cases). 
For $i=1,\dots,10$ we include a dense five-term family
\[
z^3 + x^9 + y^9 + x^2y^2w^2 + x^{i}w^5 + y^4w^{i},
\]
intended to test stability of the ranker under many mixed terms and weak tie-breaking signals.
\end{enumerate}

Taken together, these additions were designed to force long delays before any low-degree proxy changes,
create occasional spikes in the shade proxy after several canonical steps, and prevent the search from
overfitting to a single monomial ordering or to a narrow plateau regime.

\section{Evolved ranking functions in dim $4$, $p=3$}
\label{app:evolved}

For reproducibility we record the three evolved \texttt{ranking\_function} blocks discussed in Section \ref{sec:results}. They are pure functions of the $26$-tuple of features described in Section~\ref{sec:features}.
The first returns a two-component rank using a large scalarization designed to encode a lexicographic hierarchy.
The second returns an explicit five-component lex tuple which is then discretized via the map $\Pi$ of Section~\ref{subsec:discretize}. The third uses the same discretixzation map and discovers more robust lex components.  

\subsection{$\mathcal{R}(x) = (c_1(x), c_{\mathrm{combined}}(x))$: ranking function achieving perfect score on the $71$-case test set}

\begin{verbatim}
# EVOLVE-BLOCK-START
import math
from typing import Tuple, Union

Rank = Tuple[Union[int, float], ...]
Features = Tuple[Union[int, float], ...]

def ranking_function(features: Features) -> Rank:
    # Core features (old indices preserved)
    f_max_order = float(features[0])               # f0
    f_elim_order = float(features[1])              # f1
    f_boundary_count = float(features[4])          # f4
    f_shade = float(features[5])                   # f5
    f_newton = float(features[7])                  # f7
    f_E_order = float(features[8])                 # f8
    f_monomial_phase = int(features[9])            # f9
    f_insep_init = float(features[10])             # f10
    f_word = float(features[14])                   # f14
    f_wild = float(features[18])                   # f18

    # New appended features
    f_base_dim_locus = float(features[19])         # f19
    f_base_comp_locus = float(features[20])        # f20
    f_HS_base = float(features[21])                # f21
    f_jac_min_ord = float(features[22])            # f22
    f_jac_nonzero = float(features[23])            # f23
    f_padic_depth = float(features[24])            # f24
    f_boundary_mult_sum = float(features[25])      # f25

    # Component 1: Monomial Boundary Rule + Primary f0 dependence.
    c1 = 0.0 if f_monomial_phase == 1 else (f_max_order + 0.25)

    # Component 2: Primary Weighted Order Descent.
    c2_orig_val = f_word

    # Component 3: Hilbert-Samuel, Inseparable Initial Forms, and Jacobian Signal with Locus Penalty.
    jacobian_reward_term = (1.0 - f_jac_nonzero) * (1.0 + f_jac_min_ord / 5.0)
    locus_penalty_term = 0.1 * f_base_dim_locus + 0.05 * f_base_comp_locus

    # Phase-Space Potential Mapping (HS, p-adic) -> potential
    x_hs = f_HS_base / 25.0
    y_padic = f_padic_depth / 10.0
    phase_potential = -5.0 * math.atan2(y_padic, x_hs)

    aggregate_plateau_input = (
        f_HS_base
        + locus_penalty_term
        + f_insep_init
        - jacobian_reward_term
        + phase_potential
    )
    c3_orig_val = math.tanh(aggregate_plateau_input / 5.0) * 50.0

    # Component 4: Primary Descent Accelerators.
    c4_orig_val = (
        f_elim_order * 0.15
        - f_newton * 1.5
        - 1.0 * math.exp(f_boundary_mult_sum * 0.1)
        + f_E_order * 0.2
    )

    # Component 5: Wildness and Residual Boundary Information.
    c5_orig_val = (
        f_shade * 0.5
        + f_wild * 0.5
        - f_boundary_count * 0.1
    )

    # Scalarization weights (chosen to encode lex hierarchy)
    MAX_ABS_C5 = 55.0
    MAX_ABS_C4 = 22326.0
    MAX_ABS_C3 = 50.0
    SAFETY_MARGIN = 1.0

    W_c5 = 1.0
    W_c4 = 250.0
    W_c3 = math.ceil(MAX_ABS_C4 + SAFETY_MARGIN) * W_c4
    W_c2 = math.ceil(MAX_ABS_C3 + SAFETY_MARGIN) * W_c3

    c_combined = (
        W_c2 * c2_orig_val
        + W_c3 * c3_orig_val
        + W_c4 * c4_orig_val
        + W_c5 * c5_orig_val
    )
    return (c1, c_combined)

# EVOLVE-BLOCK-END
\end{verbatim}

\subsection{$\mathcal{R}_{disc}$: discretized lex ranking function achieving perfect score on the 71-case suite}

\begin{verbatim}
# EVOLVE-BLOCK-START
from typing import Tuple, Union
import math

Rank = Tuple[Union[int, float], ...]
Features = Tuple[Union[int, float], ...]

def ranking_function(features: Features) -> Rank:
    # Core features (old indices preserved)
    f_max_order = float(features[0])               # f0
    f_elim_order = float(features[1])              # f1
    f_boundary_count = float(features[4])          # f4
    f_shade = float(features[5])                   # f5
    f_newton = float(features[7])                  # f7
    f_E_order = float(features[8])                 # f8
    f_monomial_phase = int(features[9])            # f9
    f_insep_init = float(features[10])             # f10
    f_word = float(features[14])                   # f14
    f_wild = float(features[18])                   # f18

    # New appended features
    f_base_dim_locus = float(features[19])         # f19
    f_base_comp_locus = float(features[20])        # f20
    f_HS_base = float(features[21])                # f21
    f_jac_min_ord = float(features[22])            # f22
    f_jac_nonzero = float(features[23])            # f23
    f_padic_depth = float(features[24])            # f24
    f_boundary_mult_sum = float(features[25])      # f25

    # Component 1: must be 0 in monomial phase, >0 otherwise.
    c1 = 0.0 if f_monomial_phase == 1 else float(f_max_order)

    # Component 2: primary plateau breaker (linearized).
    c2 = (
        0.5 * f_word
        + 0.5 * f_HS_base
        + 0.05 * f_elim_order
        + 0.01 * f_shade
    )

    # Component 3: secondary locus/inseparability signals.
    c3 = (
        f_insep_init
        + f_base_dim_locus
        + 0.1 * f_base_comp_locus
    )

    # Component 4: "cubic depth charge" (log-scale target).
    pathology_interaction_amplifier = f_insep_init * f_padic_depth * (1.0 - f_jac_nonzero)
    c4 = -1.0 * (
        4.0 * (f_padic_depth ** 3)
        + 1.0 * f_boundary_mult_sum
        + 5.0 * (1.0 - f_jac_nonzero) * f_padic_depth
        + 10.0 * pathology_interaction_amplifier
    )

    # Component 5: final tie-breaker.
    c5 = f_wild + 0.5 * f_E_order

    return (c1, c2, c3, c4, c5)

# EVOLVE-BLOCK-END
\end{verbatim}

\subsection{$\mathcal{R}_{100}$: final robust ranking function}

\begin{verbatim}

# EVOLVE-BLOCK-START
import math
from typing import Tuple, Union

Rank = Tuple[Union[int, float], ...]
Features = Tuple[Union[int, float], ...]

def ranking_function(features: Features) -> Rank:
    """
    Candidate rank function.

    IMPORTANT: feature indices 0..18 are the original features,
    and new features were appended after that (19+).
    This keeps backward compatibility with older rankers.
    MUST BE PURE: no global state, no history, no randomness.
    """

    # Extract common features for readability and performance
    f0  = float(features[0])   # Max order
    f1  = float(features[1])   # Elimination order
    f4  = float(features[4])   # Boundary count
    f5  = float(features[5])   # Shade penalty
    f6  = float(features[6])   # Jacobian vanish flag
    f7  = float(features[7])   # Newton slope
    f8  = float(features[8])   # E_order proxy
    f9  = int(features[9])     # Monomial phase flag
    f10 = float(features[10])  # Inseparable initial flag
    f12 = float(features[12])  # primary_jac_problem
    f13 = float(features[13])  # secondary_jac_problem
    f14 = float(features[14])  # Weighted order proxy
    f15 = float(features[15])  # Order shift penalty
    f17 = float(features[17])  # Weighted multiplicity deviation
    f18 = float(features[18])  # Wildness index
    f19 = float(features[19])  # Dim of singular locus
    f20 = float(features[20])  # Num components
    f21 = float(features[21])  # Hilbert-Samuel
    f22 = float(features[22])  # Jacobian order proxy
    f23 = float(features[23])  # Jacobian nonzero partials proxy
    f24 = float(features[24])  # p-adic depth
    f25 = float(features[25])  # Boundary sum

    # Component 1: Phase/Order
    c1 = 0.0 if f9 == 1 else f0

    # Component 2: Primary Complexity Minimization.
    c2 = (
        1.0 * f14
        + 0.1 * f21
        + 0.1 * f1
        + 0.8 * f23
        + 0.5 * f7
        + 0.2 * f17
    )

    # Component 3: Structural Complexity (Bounded).
    c3 = (
        f10
        + 2.0 * f19
        + 0.5 * f20
        + 0.1 * f4
        + 0.2 * f12
        + 0.1 * f13
    )

    # Component 4: Accumulators and contextual pathology interaction.
    baseline_acc_term = (
        10.0 * (f24 ** 2.0)
        + 5.0 * f25
    )

    jac_problem_signal = (
        f6
        + (1.0 - f23)
        + f12
        + f13
    )
    jac_problem_activation = max(0.0, math.tanh(jac_problem_signal + f21 / (1.0 + f22)))

    insep_wild_signal = (
        f10
        + f18
        + f5
        + f19 * f20
        + f4
    )
    insep_wild_activation = max(0.0, math.tanh(insep_wild_signal / 5.0))

    effort_pressure = max(0.0, math.tanh((f24 + f25 + f4) / 10.0))

    catastrophe_magnitude_coeff = 1000.0 * (1.0 + math.tanh((baseline_acc_term + f4 + f5 + f18) / 100.0))

    combined_pathology_activation_strength = (
        0.01
        + 0.5 * jac_problem_activation
        + 0.5 * insep_wild_activation
        + 0.1 * effort_pressure
    )

    contextual_pathology_interaction_term = catastrophe_magnitude_coeff * math.exp(combined_pathology_activation_strength)

    c4 = -1.0 * (baseline_acc_term + contextual_pathology_interaction_term)

    # Component 5: Tie-breaker.
    c5 = (
        f18
        + f5
        + 0.5 * f8
        + 2.0 * f6
        + 0.1 * f15
        - 0.1 * f22
    )

    return (c1, c2, c3, c4, c5)

    # EVOLVE-BLOCK-END
\end{verbatim}

\bibliographystyle{abbrv}
\bibliography{resolutions}

\end{document}